\documentclass[a4,11pt]{article}
\usepackage{amssymb, amsmath, latexsym}
\usepackage[margin=3.5cm]{geometry}
\usepackage[titletoc,title]{appendix}
\numberwithin{equation}{section}
\newtheorem{theorem}{Theorem}
\newtheorem{lemma}{Lemma}
\newtheorem{proposition}{Proposition}

\newtheorem{remark}{Remark}
\newtheorem{proof}{Proof}

\title
 {On the Initial Data Constraints on the Light cone for the Einstein-Vlasov system}
\author{Patenou Jean Baptiste}
\date{}
\begin{document}
\maketitle  \abstract{ This article is concerned with the derivation
of the Gauss-Codazzi's constraints
  equations on the initial light cone for geometric
transport equations in general relativity. Temporal-gauge-dependent
constraints are addressed too and gauge-preservation is established.
The global resolution of the constraints is studied, a large class
of initial data sets is deduced from appropriate free data and their
behavior at the vertex of the cone is examined.}


\section{Introduction and various issues}
It is well known that by considering the characteristic Cauchy
problem in general relativity, the treatment of the initial data
constraints problem is considerably simplified as the constraints
reduce to propagations equations along null geodesics generating the
initial manifold carrying the data provided free data are
well-chosen
\cite{a},\cite{l},\cite{f},\cite{v},\cite{s},\cite{n},\cite{u},\cite{i},\cite{t},\cite{j},\cite{k}.
This is a particular interesting feature of the characteristic
Cauchy problem in general relativity opposed to the ordinary
spacelike Cauchy problem where the constraints equations are of
elliptic type. However, a full description of the constraints in the
characteristic Cauchy problem setting requires a priori
investigation of the geometry of the null initial hypersurfaces in
consideration and derivation on it of explicit expressions of
various null geometry quantities
\cite{l},\cite{f},\cite{v},\cite{s},\cite{n},\cite{u},\cite{i},\cite{t},\cite{j},\cite{k}.
Furthermore, while in the case of the ordinary spacelike Cauchy
problem in general relativity the constraints are standard, depend
only of the geometric nature of the Einstein's equations (i.e. the
so called Hamiltonian and momentum constraints) \cite{m},\cite{b},
in the case of the characteristic initial value problem, one faces
the difficult task of highlighting additional gauge-dependent
constraints essential to the construction of the full set of initial
data
\cite{l},\cite{f},\cite{v},\cite{s},\cite{n},\cite{u},\cite{i},\cite{t},\cite{j},\cite{k}.
These latter are induced by: the
 evolution system deduced of the splitting of the Einstein equations
   by the choice of a gauge, the form of the stress-energy
    momentum tensor of the matter involved, and their hierarchy
 depends heavily on the prescribed free data.
  On the other hand recent developments in the direction
of the study of smoothness of Scri \cite{v} besides the challenge
for the global existence theory in general relativity
\cite{p},\cite{a},\cite{z},\cite{e},\cite{r} indicate the importance
to enlarge the approach of the treatments of the characteristic
initial data constraints problem in general relativity. As instance,
in \cite{v},\cite{n}, it is revealed that harmonic gauge is not
appropriate to tackle some difficulties related to the occurrence of
log-terms in the constraints at infinity. In this paper, the new
approach consists to investigating the temporal gauge in the
characteristic Cauchy problem setting. This requires that the shift
is null and the time is in wave gauge. The characteristic initial
value problem for the Einstein-Vlasov system on the light cone
splits in the characteristic Cauchy problem on the light cone for
the evolution system (\ref{x})-(\ref{xx}) and the initial data
constraints problem on the light cone for the Einstein-Vlasov
system. This gauge helps in particular in the treatment of the
initial data constraints's problem on a light cone when kinetic
matter is involved. Indeed, the presence of all the components of
the metric in each component of the momentum tensor of matter due to
the Vlasov's field makes difficult the use of the Rendall's scheme
of resolution of the initial data constraints's problem. Such
difficulties are revealed in \cite{c},\cite{n},\cite{k}. The
interest for the characteristic Cauchy problem in general relativity
is well known \cite{j}, and there is a growing interest for this
since the work of D. Christodoulou on the formation of Black holes
in general relativity \cite{w}. The gauge mostly used in this
context is the harmonic or wave gauge (its generalization is the
"generalized wave map gauge")
\cite{c},\cite{l},\cite{f},\cite{v},\cite{z},\cite{n},\cite{j},\cite{k},
which fits well to some types of matter. Another gauge ie. the
"Double null foliation gauge" is now also used and principally in
vacuum \cite{a},\cite{h},\cite{y}. For all these gauges, the
question of existence of global solutions for the constraints's
equations is of great interest \cite{f},\cite{v},\cite{h}.
  The rest of the paper is structured as follow: first we recall the
  evolution system (\ref{x})-(\ref{xx}) induced by the choice of the temporal
  gauge, thereby we identify the Cauchy data to be attached to this
  system and the type of constraints on concerned, then we analyze
  the constraints in coordinates adapted to the null geometry of the
  cone, this yields a full description of the constraints followed
  by their resolution from appropriate free data. The behavior of
  the solutions of the constraints at the tip of the cone is
  analyzed. The preservation of the gauge is established.
   The question left open is the study of
the class of free data on a cone which leads to a smooth solution of
the Einstein-Vlasov system on a neighborhood of the vertex of the
cone.

\bigskip

\section{The setting and the evolution system}
The Einstein equations of general relativity are geometric in nature
and do not
 take a specific partial differential equations type, unless a well-chosen of
 gauge is introduced. They describe the gravitational potential $g$.
 The Vlasov equation in turn appears in
 kinetic theory, it governs the density $\rho$ of moving particles. We are thus
 interested here in the characteristic Cauchy problem in a domain $Y_O$ above the light cone of vertex O
 for the combination of these
 equations, this models a spacetime $(Y_O,g)$ with collisionless matter, with $Y_O$ a Lorentzian manifold.
In a global set of coordinates
$(x^\alpha)=(x^0,x^1,x^a),\;(\alpha=0,1,...,n;\;a=2,...,n)$ of
$\mathbb{R}^{n+1}=\mathbb{R}^2\times \mathbb{R}^{n-1}, (n\geq 3)$,
these equations read:
\begin{eqnarray}\label{x7}
  H_g:\; G_{\mu\nu}\equiv R_{\mu\nu}-\frac{R}{2}g_{\mu\nu}&=&
  T_{\mu\nu},\\\label{y2}
  H_\rho:\;p^\alpha\frac{\partial\rho}{\partial x^\alpha}-
    \Gamma_{\mu\nu}^i p^\mu p^\nu\frac{\partial\rho}{\partial p^i}
    &=&0.
  \end{eqnarray}One considers that the
particles are of rest mass $\textbf{m}$, and move towards the future
$(p^0>0)$ in their mass shell
\begin{equation}
\mathbb{P}:=\{(x^\delta,p^\mu)\in Y_O\times
\mathbb{R}^{n+1}/\;g_{\mu\nu}p^\mu p^\nu=-\textbf{m}^2,\;p^0>0\}.
\end{equation}
 The terms $R_{\mu\nu},\;R$ and $G_{\mu\nu}$ design
respectively the components of the Ricci tensor, the scalar
curvature, the components of the Einstein tensor $G$, relative to
the searched metric $g$, while the $T_{\mu\nu}$ are the components
of the stress-energy momentum tensor of matter. The
$\Gamma^\lambda_{\mu\nu}$ are the Christoffel symbols of $g$, the
$p^\lambda$ stand as the components  of the momentum of the
particles w.r.t. the basis $(\frac{\partial}{\partial x^\alpha})$ of
the fiber $\mathbb{P}_x:=\{(p^\alpha)\in
\mathbb{R}^{n+1}/g_{\mu\nu}(x)p^\mu p^\nu=-\textbf{m}^2,\;p^0>0\}$
of $\mathbb{P}$, and
\begin{equation}\label{e4}
       T_{\alpha\beta}=-
    \int_{\{g(p,p)=-\textbf{m}^2\}}\frac{\rho(x^\nu, p^\mu)p_\alpha
  p_\beta\sqrt{|g|}}{p^0}\;dp^1...dp^n.
    \end{equation}As already mentioned, we investigate the temporal
    gauge \cite{m},\cite{b} requiring that
    \begin{equation}\label{y1}
g_{0i}=0,\;\Gamma^0\equiv
g^{\lambda\delta}\Gamma_{\lambda\delta}^0=0,i=1,...,n;
\lambda,\delta=0,...,n.
    \end{equation}
     Following Choquet Bruhat, the evolution system
$(H_{\overline{g}},H_\rho) $ attached to the Einstein-Vlasov system
$(H_{g},H_\rho) $ and induced by this gauge
 \cite{m},\cite{b} is
\begin{eqnarray}\label{x}
H_{\overline{g}}:\;\partial_0
R_{ij}-\overline{\nabla}_iR_{j0}-\overline{\nabla}_jR_{i0}&=&\partial_0
\Lambda_{ij}-\overline{\nabla}_i\Lambda_{j0}-\overline{\nabla}_j\Lambda_{i0};
\\\label{xx}
  H_\rho:\;p^\alpha\frac{\partial\rho}{\partial x^\alpha}-
    \Gamma_{\mu\nu}^i p^\mu p^\nu\frac{\partial\rho}{\partial p^i}
    &=&0;
\end{eqnarray}with $\Lambda_{\mu\nu}=T_{\mu\nu}+\frac{g^{\lambda\delta}T_{\lambda\delta}}{1-n}g_{\mu\nu}$,
and where the system $H_{\overline{g}}$ replaces the Einstein
equations
 and its principal part is $ \square \partial_0
  \overline{g}_{ij}$, $
\overline{\nabla}$ denotes the connection w.r.t. the induced metric
$\overline{g}$ on $ \Lambda_t: x^0=t$. In this paper, attention is
focused on the construction and resolution
    of the constraints satisfied by a large class of initial data
$(\overline{g}_0,k_0,\rho_0)$ on $\mathcal{C}\times \mathbb{R}^n$
with $\mathcal{C}$ of equation
\begin{equation}\label{x2}
\mathcal{C}:\;x^0-r=0,\;r:=\sqrt{\sum_{i=1}^n(x^i)^2},
\end{equation}
s.t. for any solution $(\overline{g},\rho)$ of the evolution system
$(H_{\overline{g}},H_\rho)$ on $\mathbb{P}$ satisfying
$\overline{g}_{|\mathcal{C}}=\overline{g}_0,\;(\partial_0\overline{g})_{|\mathcal{C}}=k_0,\;
\rho_{|\mathcal{C}}=\rho_0$, $(g,\rho)$ is solution of the
Einstein-Vlasov system $ (H_{g},H_\rho)$ in $\mathbb{P}$, where $g$
is of the form
\begin{equation}\label{e1}
    g=-\tau^2(dx^0)^2+\overline{g}_{ij}dx^idx^j,
\end{equation} with $\tau^2=(c(x^i))^2|\overline{g}|$, and
$c$ is a positive scalar density on $ \Lambda_t$ which is determined
by the prescribed data such that $\Gamma^0=0$ in $Y_O$.
\section{The characteristic initial data constraints}
In what follows, one sets $q_s=-\frac{x^s}{r},
q^s=g^{sl}q_l,\;X^{\mu\nu}\equiv G^{\mu\nu}-T^{\mu\nu}$. The start
point of the construction of the constraints is the following
proposition.
 \begin{proposition}\label{p1}
 For any $\mathcal{C}^\infty$ solution
$(\overline{g},\rho)$ in $ \mathbb{P}$ of the evolution system
(\ref{x})-(\ref{xx}) such that with respect to the metric $g$ of the
form (\ref{e1})
$X^{\mu\nu}_{/\mathcal{C}}=0,\;(\partial_0X_{0s})(O)=0; s=1,...,n,
s\neq {s_0},q^{s_0}(O)\neq 0 $, $(g,\rho)$
 is a solution in $ \mathbb{P}$ of the Einstein-Vlasov
 system.\end{proposition}
\begin{proof}
  If $(\overline{g},\rho)$ is a $\mathcal{C}^\infty$
  solution in $ \mathbb{P}$ of the evolution system
(\ref{x})-(\ref{xx}), then with respect to the metric $g$ of the
form (\ref{e1}) tied to $\overline{g}$ and in virtue of Bianchi
identities, the tensor $(X^{\mu\nu})$ satisfies the homogeneous
linear Leray-hyperbolic system ( "see"
 \cite{b}, pages [407-414]):
           \begin{eqnarray}\label{e12}
             \partial_0 X^{00}+L^{00}(X^{\gamma\alpha},\partial_i
X^{i0})&=&0 \\\label{e13}
            \partial_0 X^{ij}+L^{ij}(X^{\gamma\alpha},
\partial_sX^{k0})&=&0  \\\label{e14}
   \square_g X^{0j}
+L^{0j}(X^{\gamma\alpha},\partial_sX^{\delta\beta})&=&0.
           \end{eqnarray}Another homogeneous third order Leray hyperbolic system derived from this one is the
           system
           \begin{equation}\label{y3}
                \square_g\partial_0X^{\lambda\delta}+T^{\lambda\delta}(X^{\alpha\beta},D^\epsilon
                X^{\mu\nu})=0,\;|\epsilon|\leq 2.
           \end{equation}Now if one has
           $X^{\mu\nu}_{/\mathcal{C}}=0$, then on $\mathcal{C}$, $[\partial_0 X^{00}],\;[\partial_0
           X^{ij}]$ express in terms of
           $[\partial_0X^{0i}],\;i=1,...,n$ as a consequence of
           restriction to $\mathcal{C}$ of the equations
           (\ref{e12})-(\ref{e13}).
Substituting these expressions in the system (\ref{e14}) restricted
to $\mathcal{C}$, this latter appears then in turn as a homogeneous
linear differential system of propagation equations on $\mathcal{C}$
of unknowns $[\partial_0X^{0i}],\;i=1,...,n$. Now, thanks once more
to
 the Bianchi relations $\nabla_\alpha X^{\alpha i}=0,\;i=1,...,n$
  and the evolution system $H_{\overline{g}}$ written at O,
 it suffices that $\partial_0X_{0k}(O)=0,\;k=1,...,n,\;k\neq k_0,\; q^{k_0}\neq 0$ so that
 $(\partial_0 X^{\mu\nu})(O)=0$. Finally $X^{\mu\nu}=0$ in $Y_O$
thanks to (\ref{y3}).\end{proof}
\subsection{Null geometry on the cone}
The above proposition \ref{p1} is an indication that the Einstein
equations must hold on the initial hypersurface $\mathcal{C}$ in
order that the preservation of the gauge is established. This
induces Gauss and Codazzi's constraints on $\mathcal{C}$, but these
standard constraints must be supplemented by other gauge-dependent
constraints. To carry out the
    analysis, we introduce  null adapted coordinates
w.r.t. the trace of the metric on the cone $\mathcal{C}$
$(y^0,y^1,y^A)$ (\cite{c},\cite{l},\cite{f},\cite{v}) defined by
\begin{equation}\label{c2}
    y^0=x^0-r,\; y^1=r=\sqrt{\sum_{i=1}^n(x^i)^2},\;y^A,\;A=2,...,n;
    \end{equation}where
$(y^A)$ design local coordinates in the sphere $\mathcal{S}^{n-1}:
\sum_{i=1}^n(x^i)^2=1$, then
    \begin{equation}\label{c3}
        x^0=y^0+y^1,\;x^i=y^1\theta^i(y^A),\;\sum_i(\theta^i)^2=1;
    \end{equation}and the $\theta^i(y^A)$ are $\mathcal{C}^\infty$ functions of
    $y^A,\;A=2,...,n$;
and require the
     assumption:
    \begin{equation}\label{x3}
        \textbf{(A)}: \hbox{The vector fields $\frac{\partial}{\partial y^1}$ is tangent to the null geodesics
        generating $\mathcal{C}$ },
    \end{equation} which is an aspect of the affine parametrization condition of \cite{c},\cite{l},\cite{f},\cite{v},
     (since no a priori condition is given on the non-affinity constant on $\mathcal{C}$),
      and is equivalent to the requirement that on the cone the lapse $\tau^{-2}$ is
    an eigenvalue of the Riemannian metric $\overline{g}=(g^{ij})$, the
    corresponding eigenvector being $-q=(-q_i),\;q_i=-\frac{x^i}{r}$. The components of tensors in coordinates $(y^\mu)$
    are equipped with a tilde $"\;\widetilde{}\;"$. The assumption
    $\textbf{(A)}$ induces that the trace on the null cone
    $\mathcal{C}:y^0=0$,
    of the searched metric in temporal gauge is of the form
\begin{equation}\label{c4}
    g_{_{|C}}= \widetilde{g}_{01}dy^0dy^0+\widetilde{g}_{01}(dy^0dy^1+dy^1dy^0)+
    \widetilde{g}_{AB}dy^Ady^B,\;A,B=2,...,n.
\end{equation}The non-zero components of the trace on $\mathcal{C}$ of the inverse $g^{-1}$ of
$g$ satisfy:
\begin{equation}\label{y4}
    \widetilde{g}^{01}\widetilde{g}_{01}=1,\;
    \widetilde{g}^{11}+\widetilde{g}^{01}=0,\;(\widetilde{g}^{AB})=(\widetilde{g}_{AB})^{-1},\;A,B=2,...,n.
\end{equation}We remark that $\widetilde{g}_{01}<0$ according to the signature of $g$.
 Using the expression (\ref{c4}) of the trace on $\mathcal{C}$ of the metric
and its inverse (\ref{y4}), one derives, some algebraic relations
between the gravitational data on $\mathcal{C}$, the expressions of
the restrictions on $\mathcal{C}$ of the Christoffel symbols of the
metric. The details are in appendix \ref{A2}. One can then derive
the restrictions on $\mathcal{C}$ of the components of
$X=(X_{\mu\nu})$.
\subsection{Gauss and Codazzi's constraints on $\mathcal{C}$}
The vectorfield $l=(0,-\widetilde{g}^{01},0,...,0)$ is outgoing
normal to $\mathcal{C}$, the projection operator on $\mathcal{C}$ is
 $\pi$, s.t. $\pi_\mu^\nu=\delta_\mu^\nu+l_\mu l^\nu $, then the Gauss and Codazzi's constraints on the light cone
$\mathcal{C}$ read
\begin{equation}\label{y5}
    \widetilde{X}_{\mu\nu}l^\mu l^\nu=0,\;\widetilde{X}_{\mu\nu}l^\mu\pi^\nu_\lambda=0,
\end{equation}and resume to
\begin{equation}\label{y6}
    \widetilde{X}_{1\lambda}=0,\;\lambda=0,1,...,n.
\end{equation}
These constraints involve naturally only the Cauchy data for the
evolution system $(H_{\overline{g}},H_\rho)$. The other Einstein
equations
\begin{equation}\label{x6}
    \widetilde{X}_{00}=0,\;\widetilde{X}_{0A}=0,\;\widetilde{X}_{AB}=0,\;A,B=2,...,n;
\end{equation}
do not play the role of constraints as their expressions on
$\mathcal{C}$ contain second order outgoing derivatives of the
metric which are not part of the initial data of the third order
characteristic problem for the evolution system
$(H_{\overline{g}},H_\rho)$. Indeed, one has from straightforward
computations the following expressions which reveal the harmful
terms that one faces:
\begin{eqnarray}\label{x12}
  \widetilde{X}_{00} &=& \frac{1}{2}\widetilde{g}^{01}
   \widetilde{g}^{AB}\partial^2_{00}\widetilde{g}_{AB}+H_1(\widetilde{g},\partial\widetilde{g},\rho), \\
   \label{y13}
 \widetilde{X}_{AB} &=& \frac{1}{2}
   (\widetilde{g}^{01})^2\partial^2_{00}\widetilde{g}_{11}\widetilde{g}_{AB}+
    H_2(\widetilde{g},\partial\widetilde{g},\rho), \\\label{y14}
 \widetilde{X}_{0A}&=&-\frac{1}{2}\widetilde{g}^{01}\partial^2_{00}\widetilde{g}_{A1}+
    H_3(\widetilde{g},\partial\widetilde{g},\rho).
\end{eqnarray}
\subsection{Temporal gauge-dependent constraints}
In this subsection we are led to finding modifications or
combinations of the Einstein equations (\ref{x6}) in order to
construct the gauge-dependent constraints, and this is lengthy and
somewhat subtle. The Einstein equations $\widetilde{X}_{AB}=0$ read:
\begin{equation}\label{y15}
        \widetilde{X}_{AB}\equiv \widetilde{R}_{AB}-\frac{1}{2}\widetilde{g}_{AB}
    (2\widetilde{g}^{01}\widetilde{R}_{01}+\widetilde{g}^{11}\widetilde{R}_{11}+
    \widetilde{g}^{CD}\widetilde{R}_{CD})-\widetilde{T}_{AB}=0.
    \end{equation}They comprise the term $\widetilde{R}_{01}$ which according to the expressions of the components of
    the Ricci tensor on $\mathcal{C}$ ("see" appendix \ref{A3}), has second order outgoing derivative of the
    metric, ie:
        \begin{equation*}
        \widetilde{R}_{01}=-\frac{1}{2}
   (\widetilde{g}^{01})^2\partial^2_{00}\widetilde{g}_{11}+
   H_4(\widetilde{g},\partial\widetilde{g}).
    \end{equation*}The treatment of this term induces the equations
    \begin{equation*}
    \widetilde{X}_{AB}-\frac{\widetilde{g}^{CD}\widetilde{X}_{CD}}{n-1}\widetilde{g}_{AB}=0,\;
    A,B,C,D=2,...,n,
    \end{equation*}
    which are equivalent to
    \begin{equation*}
        \widetilde{R}_{AB}-\widetilde{T}_{AB}-\frac{\widetilde{g}^{CD}
(\widetilde{R}_{CD}-\widetilde{T}_{CB})}{n-1}\widetilde{g}_{AB}=0
,\;
    A,B,C,D=2,...,n.
    \end{equation*}
 These latter equations do not contain second order outgoing derivatives
of the metric. The Einstein equation $\widetilde{X}_{00}=0$ in turn
reads:
\begin{equation}\label{y16}
        \widetilde{X}_{00}\equiv \widetilde{R}_{00}-\frac{1}{2}\widetilde{g}_{00}
    (2\widetilde{g}^{01}\widetilde{R}_{01}+\widetilde{g}^{11}\widetilde{R}_{11}+
    \widetilde{g}^{CD}\widetilde{R}_{CD})-\widetilde{T}_{00}=0.
    \end{equation}Second order outgoing derivatives in this equation are due to
     the terms $\widetilde{R}_{00}$ and $\widetilde{R}_{01}$ since
     ("see" appendix \ref{A3})
     \begin{equation*}
        \widetilde{R}_{00}=-\frac{1}{2}
   (\widetilde{g}^{01})^2\partial^2_{00}\widetilde{g}_{11}+
   \frac{1}{2}
   (\widetilde{g}^{AB})^2\partial^2_{00}\widetilde{g}_{AB}+
   H_5(\widetilde{g},\partial\widetilde{g}).
     \end{equation*}Dealing with these harmful terms led us to using
     of the expression
\begin{equation}\label{x14}
        \frac{\partial}{\partial y^0}(\widetilde{\Gamma}^0+\widetilde{\Gamma}^1) =\frac{1}{2}
   (\widetilde{g}^{01})^2\partial^2_{00}\widetilde{g}_{11}-\frac{1}{2}\widetilde{g}^{01}
   \widetilde{g}^{AB}\partial^2_{00}\widetilde{g}_{AB}+H_6(\widetilde{g},\partial\widetilde{g}),
\end{equation}which reflects the "time in wave gauge" property. This
analysis results to considering the equation
\begin{equation*}
\widetilde{
X}_{00}-\widetilde{g}_{01}\frac{\widetilde{g}^{CD}\widetilde{X}_{CD}}{n-1}+
   \widetilde{g}_{01}\frac{\partial (\widetilde{\Gamma}^0+\widetilde{\Gamma}^1)}{\partial
   y^0}   =0,\;A,B,C,D=2,...,n,
\end{equation*}which does not contain second order outgoing derivatives of
the metric and is equivalent to
   \begin{equation*}
   \widetilde{R}_{00}-\widetilde{T}_{00}-
   \widetilde{g}_{01}\frac{\widetilde{g}^{AB}(\widetilde{R}_{AB}-\widetilde{T}_{AB})}{n-1}+
   \widetilde{g}_{01}\frac{\partial (\widetilde{\Gamma}^0+\widetilde{\Gamma}^1)}{\partial
   y^0}   =0.
   \end{equation*}We remark that the Einstein equations $\widetilde{X}_{0A}=0$
    are not involved in the construction of the gauge-dependent constraints and their
    realization on the cone $\mathcal{C}$ will be obtained in a more indirect way.
 The construction of constraints in this section thus ends up by the following theorem.
 \begin{theorem}\label{th1} Let
$(\overline{g},\rho)$ be any $\mathcal{C}^\infty$ solution of the
evolution system $(H_{\overline{g}},H_\rho)$
 in a neighborhood $\mathcal{V}$ of $\mathcal{C}\times \mathbb{R}^n$, and
 let $g$ associated to $\overline{g}$ of the form (\ref{e1})
  s.t. the temporal gauge condition is satisfied in $Y_O=\{y^0\geq 0\}$. One sets
$\widetilde{X}_{\mu\nu}\equiv
\widetilde{G}_{\mu\nu}-\widetilde{T}_{\mu\nu}$, and
     one assumes that w.r.t. the metric $g$, the hypothesis $\textbf{(A)}$ (\ref{x3}) and the
     relations
     \begin{equation}\label{x4}
        \widetilde{X}_{1\lambda}=0,
   \lambda=0,...,n,
     \end{equation}
     \begin{equation}\label{y7}
\widetilde{X}_{AB}-\frac{\widetilde{g}^{CD}\widetilde{X}_{CD}}{n-1}\widetilde{g}_{AB}=0,
       \end{equation}
\begin{equation}\label{x5}
\widetilde{
X}_{00}-\widetilde{g}_{01}\frac{\widetilde{g}^{CD}\widetilde{X}_{CD}}{n-1}+
   \widetilde{g}_{01}\frac{\partial (\widetilde{\Gamma}^0+\widetilde{\Gamma}^1)}{\partial
   y^0}   =0,\;A,B,C,D=2,...,n;
\end{equation}
are satisfied on $\mathcal{C}$; if furthermore one has
$X_{0k}(O)=0,\;k=1,...,n; (\partial_0X_{0s})(O)=0; s=1,...,n, s\neq
{s_0},q^{s_0}(O)\neq 0 $;
     then $(g,\rho)$ is solution of the Einstein-Vlasov system
     $(H_{g},H_\rho)$ in $Y_O$.\end{theorem}
To prove the theorem, we establish the following lemma.
     One sets $[X_{\mu\nu}]=
    {X_{\mu\nu}}_{|\mathcal{C}}$.
     \begin{lemma}\label{l1}
     For any $\mathcal{C}^\infty$ solution
$(\overline{g},\rho)$ of the evolution system
$(H_{\overline{g}},H_\rho)$
 in a neighborhood $V\times \mathbb{R}^{n+1}$ of a smooth hypersurface $\mathcal{I}\times \mathbb{R}^{n+1}$
 with $\mathcal{I}$ of equation $\mathcal{I}: x^0-\phi (x^i)=0$, and
 $g$ associated to $\overline{g}$ of the form (\ref{e1})
  s.t. the temporal gauge condition is satisfied in $V$. One sets $q_i=-\frac{\partial\phi}{\partial x^i}$.
  Then the tensor $X=(X_{\mu\nu})$ restricted to $\mathcal{I}$
  satisfies the homogeneous linear system of partial differential equations
  \begin{equation}\label{x8}
    q^j\partial_j[ X_{k0}]+q^j \partial_k [X_{j0}]-q_k
    g^{lm}\partial_l[ X_{m0}]+g^{ij}\partial_i[ X_{jk}]+A^{\mu\nu}_k[X_{\mu\nu}]=0.
\end{equation}\end{lemma}
\begin{proof}
 $(\overline{g},\rho)$ is a $\mathcal{C}^\infty$ solution
of the evolution system $(H_{\overline{g}},H_\rho)$, and according
to the divergence free properties of the Einstein tensor
$(G_{\mu\nu})$ and the stress energy momentum tensor of matter
$(T_{\mu\nu})$ of $g$ (\ref{e1}), one has:
\begin{equation}\label{a50}
    \nabla^\alpha X_{\alpha\beta}{_{|\mathcal{I}}}=0,\;
\left(\partial_0
               (R_{ij}-\Lambda_{ij})-\overline{\nabla}_{i}X_{j0}-
\overline{\nabla}_{j}X_{i0}\right){_{|\mathcal{I}}}=0.
\end{equation}
Since $g_{0i}=0=g^{0i}$, the Bianchi identities induce the following
relations:
           \begin{equation*}
                g^{\alpha 0}\nabla_0 X_{\alpha\beta}+ g^{ij}\nabla_i
                X_{j\beta}=0,\;\forall \beta.
           \end{equation*}For $\beta=k$, one has successively:
           \begin{equation*}
             g^{00}\left(\partial_0 X_{0k}-\Gamma^\alpha_{00}
             X_{\alpha k}-\Gamma^\alpha_{0k}X_{0\alpha}\right) +
             g^{ij}\left(\partial_i X_{jk}-\Gamma^\alpha_{ij}
             X_{\alpha k}-\Gamma^\alpha_{ik}X_{j\alpha}\right)= 0
             \end{equation*}
            \begin{equation}\label{a44}
             g^{00}\partial_0 X_{0k}+g^{ij}\partial_i X_{jk}
             \underset{C_k}{\underbrace{-g^{00}\Gamma^\alpha_{00}X_{\alpha k}-
             g^{00}\Gamma^\alpha_{0k}X_{0\alpha}-
             g^{ij}\Gamma^\alpha_{ij}X_{\alpha k}-g^{ij}
              \Gamma^\alpha_{ik}X_{j\alpha}}}= 0.
           \end{equation}Now on $\mathcal{I}$, one has
     \begin{equation}
             {[\partial_i X_{jk}]} = {\partial_i[X_{jk}]+q_i [\partial_0
             X_{jk}]};
           \end{equation}and the system (\ref{a44}) restricted to $\mathcal{I}$ implies:
           \begin{equation}\label{a48}
             g^{00}\partial_0 X_{0k}+q^j[\partial_0{ X_{jk}}]+
             g^{ij}\partial_i [X_{jk}]+C_k = 0;
             \end{equation}where the $C_k$ are given by the
             relations (\ref{a44}).
             Furthermore, since $X_{\mu\nu}\equiv G_{\lambda\mu}-T_{\lambda\mu}=R_{\lambda\mu}-
   \Lambda_{\lambda\mu}-\frac{1}{2}g_{\lambda\mu}(R-\Lambda)$, one has:
 \begin{equation}\label{a47}
             \partial_0 X_{jk} = \partial_0 (R_{jk}-\Lambda_{jk})-
             \frac{1}{2}(\partial_0 (R-\Lambda))g_{jk}-\frac{1}{2}(R-\Lambda)\partial_0 g_{jk}
             \end{equation}On the other hand, since
              $    R-\Lambda= g^{00}(R_{00}-\Lambda_{00})+g^{kl}(R_{kl}-\Lambda_{kl}) $, one has
    \begin{equation*}
        \partial_0 (R-\Lambda) =(\partial_0
             g^{00})(R_{00}-\Lambda_{00})+g^{00}\partial_0 (R_{00}-\Lambda_{00})+
    \end{equation*}
               \begin{equation}\label{a45}
             (\partial_0
             g^{kl})(R_{kl}-\Lambda_{kl})+g^{kl}\partial_0 (R_{kl}-\Lambda_{kl}).
           \end{equation}Considering the Bianchi identities for
           $\beta=0$, one has
            successively:
               \begin{equation*}
                    g^{00}\nabla_0 X_{00}+g^{ij}\nabla_i X_{j0} = 0,
               \end{equation*}
               \begin{equation*}
                     g^{00}(\partial_0 X_{00}-2\Gamma^\alpha_{00}X_{\alpha 0})+
             g^{ij}\nabla_i X_{j0} = 0,
               \end{equation*}
               \begin{equation*}
                    g^{00}\partial_0 [(R_{00}-\Lambda_{00})-\frac{1}{2}g_{00}(R-\Lambda)]-
             2g^{00}\Gamma^\alpha_{00}X_{\alpha 0}+g^{ij}\nabla_iX_{j0}
              = 0,
               \end{equation*}
               \begin{equation*}
                    g^{00}\partial_0( R_{00}-\Lambda_{00})-\frac{1}{2}g^{00}(\partial_0
              g_{00})(R-\Lambda)-\frac{1}{2}\partial_0
              (R-\Lambda)-
               \end{equation*}
  \begin{equation}\label{a46}
              2g^{00}\Gamma^\alpha_{00}X_{\alpha 0}+g^{ij}\nabla_i
              X_{j0}=0.
           \end{equation} Combining the relations (\ref{a45}) and
           (\ref{a46}), one has:\begin{equation}\label{a43}
             \frac{1}{2}\partial_0 (R-\Lambda) =g^{kl}\partial_0 (R_{kl}-\Lambda_{kl})-g^{ij}\nabla_i
             X_{j 0}+A,
           \end{equation}where
           \begin{equation*}
                A\equiv 2^{-1}g^{00}(\partial_0
             g_{00})(R-\Lambda)+2g^{00}\Gamma^\alpha_{00}X_{\alpha 0}+(\partial_0
             g^{00})(R_{00}-\Lambda_{00})
           \end{equation*}
\begin{equation*}
   +(\partial_0
             g^{kl})(R_{kl}-\Lambda_{kl}).
\end{equation*}
Using the relations
$R_{00}-\Lambda_{00}=X_{00}+\frac{1}{2}g_{00}(R-\Lambda),\;
R_{kl}-\Lambda_{kl}=X_{kl}+\frac{1}{2}g_{kl}(R-\Lambda),\;
R-\Lambda=\frac{2}{1-n} X:=\frac{2}{1-n}g^{\lambda\delta}
X_{\lambda\delta},$, $A$ reads:
\begin{equation}\label{y9}
    A= \frac{1}{1-n}g_{kl}(\partial_0
             g^{kl})g^{\lambda\delta}X_{\lambda\delta}+2g^{00}\Gamma^\alpha_{00}X_{\alpha 0}+(\partial_0
             g^{00})X_{00}+(\partial_0
             g^{kl})X_{kl}
\end{equation}
 The system (\ref{a47}) now reads:
 \begin{equation*}
    \partial_0 X_{jk}=\partial_0 (R_{jk}-\Lambda_{jk})-g_{jk}
          \left\{g^{lm}\partial_0 (R_{lm}-\Lambda_{lm})-g^{lm}\nabla_l X_{m0}\right\}
 \end{equation*}
          \begin{equation}
          -A g_{jk}-\frac{(\partial_0 g_{jk})g^{\lambda\delta} X_{\lambda\delta}}{1-n}.
           \end{equation}The system (\ref{a48}) can then be written:
           \begin{equation*}
                g^{00}\partial_0 X_{0k}+
           \end{equation*}
           \begin{equation*}
             q^j\left(\partial_0 (R_{jk}-\Lambda_{jk})-g_{jk}g^{lm}\partial_0
      (R_{lm}-\Lambda_{lm})+g_{jk}g^{lm}\nabla_l X_{m0}-A g_{jk}\right) +
           \end{equation*}
   \begin{equation}\label{a50}
    g^{ij}\partial_i [X_{jk}]+ C_k-\frac{g^{\lambda \delta}X_{\lambda\delta}}{1-n}(\partial_0 g_{jk})q^j = 0.
   \end{equation}From the hypotheses $(\overline{g}_{ij},\rho)$
   satisfies
           \begin{equation}
               \left(\partial_0
               (R_{ij}-\Lambda_{ij})-\overline{\nabla}_{i}X_{j0}-
\overline{\nabla}_{j}X_{i0}\right){_{|\mathcal{I}}}=0,
           \end{equation}the system (\ref{a50}) becomes:
           \begin{equation*}
                g^{00}\partial_0  X_{0k}+
           \end{equation*}
           \begin{equation*}
 q^j\left\{\overline{\nabla}_j  X_{k0}+
             \overline{\nabla}_k  X_{j0}-g_{jk}g^{lm}
             (\overline{\nabla}_l X_{m0}+\overline{\nabla}_m  X_{l0})+
             g_{jk}g^{lm}\nabla_l  X_{m0}\right\}-
           \end{equation*}
 \begin{equation}\label{y12}
    A q_{k}+g^{ij}\partial_i [ X_{jk}]+C_k-\frac{g^{\lambda\delta}
X_{\lambda\delta}}{1-n}(\partial_0 g_{jk})q^j =
             0,
 \end{equation}now, given the expressions of $\nabla_l  X_{m0}$ and $\overline{\nabla}_l  X_{m0}$, one has
\begin{equation*}
    g^{00}\partial_0  X_{0k}+q^j(\partial_j [ X_{k0}]+q_j[\partial_0 X_{k0}])+
    q^j(\partial_k[ X_{j0}]+q_k[\partial_0  X_{j0}]) -
\end{equation*}
\begin{equation}\label{y11}
    q_k g^{lm}(\partial_l[ X_{m0}]+
    q_l[\partial_0
     X_{m0}])+
    g^{ij}\partial_i[ X_{jk}]+A_k= 0,
\end{equation} with:
  \begin{equation*}
    A_k\equiv C_k-2q^j\Gamma^n_{jk} X_{n0}+
    q_kg^{lm}\Gamma^s_{lm} X_{s0}-  \end{equation*}
  \begin{equation}\label{y10}
    q_k g^{lm}
    \left(\Gamma^0_{l0} X_{m0 }+
    \Gamma^0_{lm} X_{00 }+\Gamma_{l0}^sX_{ms}\right)-A q_k-
    \frac{g^{\lambda\delta} X_{\lambda\delta}}{1-n}(\partial_0 g_{jk})q^j.
  \end{equation}The terms $A$ and $C_k$ in $A_k$ are given in
  (\ref{y9}) and (\ref{a44}).\\
  Now, the hypersurface $\mathcal{I}$ being
characteristic ($g^{00}+q^jq_j=0$ on $\mathcal{I}$), the system
(\ref{y11}) resumes to
\begin{equation}
    q^j\partial_j[ X_{k0}]+q^j \partial_k [ X_{j0}]-q_k
    g^{lm}\partial_l[ X_{m0}]+g^{ij}\partial_i[ X_{jk}]+A_k=0.
\end{equation}The $A_k$
are linear combinations of $[X_{\mu\nu}]=
    {X_{\mu\nu}}_{|\mathcal{I}}$.
    \end{proof}
\begin{proof}[Proof of theorem \ref{th1}] If the hypotheses of the theorem
are satisfied together with the relations
  (\ref{x4}), (\ref{y7}), (\ref{x5}) for $g$
of the form (\ref{e1}), then straightforward calculations imply that
on $\mathcal{C}$ one has:
\begin{equation}\label{x9}
   X_{00}=\widetilde{g}_{01}\frac{\widetilde{g}^{CD}\widetilde{X}_{CD}}{n-1}=-\widetilde{g}_{01}q^sX_{0s},
    \end{equation}
    \begin{equation}\label{x32}
    X_{0k}=\widetilde{g}_{01}\frac{\widetilde{g}^{CD}\widetilde{X}_{CD}}{n-1}q_k+
    \frac{\partial y^A}{\partial x^k}\widetilde{X}_{0A},
    \end{equation}

    \begin{equation}\label{x10}
    X_{jk}=\left(q_j\frac{\partial y^A}{\partial x^k}+
    q_k\frac{\partial y^A}{\partial x^j}\right)\widetilde{X}_{0A}+
    (2\widetilde{g}_{01}q_jq_k+g_{jk})\frac{\widetilde{g}^{CD}\widetilde{X}_{CD}}{n-1},
\end{equation}
\begin{equation}\label{x11}
    X_{jk}=q_jX_{0k}+q_kX_{0j}-g_{jk}q^sX_{0s},
\end{equation}where $ \;A,B,C,D=2,...,n; i,j,k,s=1,...,n$.
 By combining the relations (\ref{x9})-(\ref{x11}) and the result (\ref{x8}) of
 lemma \ref{l1}, one obtains on $\mathcal{C}$ the homogeneous linear differential
system
\begin{equation}\label{a54}
    \frac{\partial [X_{0k}]}{\partial y^1}+
    L_k^s([X_{0s}])=0,\;k=1,...,n.
\end{equation} One deduces that if the subset
$X_{0k}=0,\;k=1,...,n$,
 of the Einstein equations are satisfied at the vertex O of the
 cone,
 then $X_{0k}=0,\;k=1,...,n$, on $\mathcal{C}$, and after that
$X_{\mu\nu}=0,\;\forall \mu,\nu$, on $\mathcal{C}$, thanks to the
relations (\ref{x9})-(\ref{x11}). The proof ends by using the
proposition \ref{p1}.
\end{proof}
 \begin{remark} We emphasize that the tensor $(Z_{AB}),\;Z_{AB}=\widetilde{X}_{AB}-
 \frac{\widetilde{g}^{CD}\widetilde{X}_{CD}}{n-1}\widetilde{g}_{AB}$ $,A,B,C,D=2,...,n$, is a traceless
 tensor,
 consequently one equation chosen suitably will not be considered while solving the constraints (\ref{y7}) and
 will be proved automatically satisfied by the traceless property.
 \end{remark}
  \section{Constraints and Cauchy data for $(H_{\overline{g}},H_\rho)$}
    The resolution of the constraints's equations requires an exhaustive description
of the constraints in terms of the Cauchy data for the evolution
system $(H_{\overline{g}},H_\rho)$, and thereby the choice of free
data. One sets:
\begin{equation*}
    \widetilde{g}_{01}{_{|\mathcal{C}}}=\theta,\;\widetilde{g}_{AB}{_{|\mathcal{C}}}=\Theta_{AB},\;
    \Theta=(\Theta_{AB}),\;\rho_{|\mathcal{C}}=\textbf{f},\;\partial_\mu=\frac{\partial}{\partial
y^\mu},
\end{equation*}
\begin{equation}\label{y17}
\psi_{\mu\nu}=
    \frac{\partial\widetilde{g}_{\mu\nu}}{\partial y^0}{_{|\mathcal{C}}}
    ,\;\pi^\alpha=\frac{\partial y^\alpha}{\partial
x^\delta}p^\delta,\;d\pi=d\pi^1...d\pi^n.
\end{equation}
\subsection{The Hamiltonian constraint on $\mathcal{C}$}
 The Hamiltonian constraint $\widetilde{X}_{11}=0$ is the so called Raychaudhuri equation.
 The equation
$\widetilde{X}_{11}=0$ reads
$\widetilde{R}_{11}-\widetilde{T}_{11}=0$. It is therefore necessary
to compute the terms $\widetilde{R}_{11}$ and $\widetilde{T}_{11}$
in order to describe this equation. From the definition of the Ricci
tensor and according to the expressions of the trace of the metric
on $\mathcal{C}$ and the corresponding Christoffel symbols ("see"
appendix \ref{A2}), on has:
\begin{equation*}
   \widetilde{ R}_{11}=\frac{1}{4}\widetilde{g}^{01}\widetilde{g}^{CB}\partial_1
\widetilde{g}_{CB}\partial_0\widetilde{g}_{11}+
\frac{1}{2}\partial_1(\widetilde{g}^{CB}\partial_1
\widetilde{g}_{CB})+\frac{1}{4}(\widetilde{g}^{CB}\partial_1\widetilde{g}_{BD})
    (\widetilde{g}^{DE}\partial_1\widetilde{g}_{CE})-
\end{equation*}
\begin{equation}\label{y22}
    \frac{1}{2}\widetilde{g}^{01}\partial_1\widetilde{g}_{01} \widetilde{g}^{CB}\partial_1
\widetilde{g}_{CB}.
\end{equation}
Relying on the definition of the stress momentum tensor of Vlasov
matter, one has:
\begin{equation}\label{y18}
    \widetilde{T}_{11}=-\int_{\mathbb{R}^n}\textbf{f}
    \frac{(\widetilde{g}_{01})^2(\pi^0)^2\sqrt{|\widetilde{g}|}}{\pi^0+\pi^1}d\pi.
\end{equation}
The mass shell equation on $\mathcal{C}\times \mathbb{R}^n$ reduces
to
\begin{equation}\label{y19}
    2\widetilde{g}_{01}\pi^0\pi^1=-(\widetilde{g}_{01}(\pi^0)^2+\widetilde{g}_{AB}\pi^A\pi^B+\textbf{m}^2).
\end{equation}On the other hand, the positivity of $g_{ij}(\lambda p^i+q^i)(\lambda
p^j+q^j)$ for every $\lambda \in \mathbb{R}^\ast$, and the fact that
the cone is characteristic imply that
\begin{equation}\label{y20}
    (p^0)^2-(q_ip^i)^2\geq \tau^{-2}\textbf{m}^2,
\end{equation}one deduces then that for a non-zero mass
$\textbf{m}$, one has $|p^0|>|q_ip^i|$, and hence $\pi^0 >0$. For a
zero mass $(\textbf{m}=0)$, one has on $\mathcal{C}\times
\mathbb{R}^n$
\begin{eqnarray*}
    (p^0)^2&=&- g^{00}g_{ij}p^ip^j=-g^{00}\left(-\widetilde{g}_{01}q_iq_j+
    \frac{\partial
y^A}{\partial x^i}
        \frac{\partial y^B}{\partial
        x^j}\widetilde{g}_{AB}\right)p^ip^j\\
        &=& (q_ip^i)^2-g^{00}\frac{\partial
y^A}{\partial x^i}
        \frac{\partial y^B}{\partial
        x^j}\widetilde{g}_{AB}p^ip^j
\end{eqnarray*}and subsequently
\begin{equation}
(p^0)^2=(q_ip^i)^2-g^{00}\widetilde{g}_{AB}\pi^A\pi^B.
\end{equation}In the case of zero mass $(\textbf{m}=0)$, one thus assumes that the support of the initial density of
particles $\textbf{f}$ is a subset of $\{(\pi^\delta)\in
\mathbb{R}^{n+1}, \;\sum_A(\pi^A)^2>c_2>0\}$, and as a consequence
of the symmetric positive definite character of the matrix
$(\widetilde{g}_{AB})$, one has $\pi^0>0$. Having $\pi^0>0$ induces
in virtue of the relation (\ref{y19}) that $\pi^1
>0$ since $\widetilde{g}_{01}<0$. The expression of $\pi^0$ on $\mathcal{C}\times
\mathbb{R}^n$ is therefore
\begin{equation}\label{y24}
    \pi^0=-\pi^1+\sqrt{(\pi^1)^2-\widetilde{g}^{01}(\widetilde{g}_{AB}\pi^A\pi^B+\textbf{m}^2)},
\end{equation}and we end up with the following expression for
$\widetilde{T}_{11}$ where
$\widetilde{\gamma}=(\widetilde{g}_{AB})$:
\begin{equation}\label{y21}
    \widetilde{T}_{11}=-\int_{\mathbb{R}^n}\textbf{f}\frac{|\widetilde{g}_{01}|^3
    (-\pi^1+\sqrt{(\pi^1)^2-\widetilde{g}^{01}(\widetilde{g}_{AB}\pi^A\pi^B+\textbf{m}^2)})^2}{
    \sqrt{(\pi^1)^2-\widetilde{g}^{01}(\widetilde{g}_{AB}\pi^A\pi^B+\textbf{m}^2)}}\sqrt{|\widetilde{\gamma}|}\;d\pi.
\end{equation}We now combine this expression of $\widetilde{T}_{11} $
 with the one of $\widetilde{R}_{11}$
(\ref{y22}) and use the notations (\ref{y17}) to obtain the
following proposition.
 \begin{proposition}\label{p2}Let $(\overline{g},\rho)$ be
any $\mathcal{C}^\infty$ solution of the evolution system
$(H_{\overline{g}},H_\rho)$
 in a neighborhood $\mathcal{V}$ of $\mathcal{C}\times \mathbb{R}^n$, and
 let $g$ associated to $\overline{g}$ of the form (\ref{e1})
  s.t. the temporal gauge condition is satisfied in $Y_O=\{y^0\geq
  0\}$. The Hamiltonian constraint $\widetilde{X}_{11}=0$ is
  satisfied on $\mathcal{C}$ if and only if the Cauchy data
  $\theta,\;\Theta=(\Theta_{AB}),\;\textbf{f},\;\psi_{11}$
  verify the partial differential relation
 \begin{eqnarray*}
    \theta^{-1}(\Theta^{AB}\partial_1\Theta_{AB})\psi_{11}+2\partial_1(\Theta^{AB}\partial_1\Theta_{AB})+
 \end{eqnarray*}
 \begin{equation*}
    (\Theta^{CB}\partial_1\Theta_{BD})(\Theta^{DE}\partial_1\Theta_{CE})-
    2\theta^{-1}
    (\Theta^{AB}\partial_1
    \Theta_{AB})\partial_1\theta=
 \end{equation*}
\begin{equation}\label{x18}
    -4\int_{\mathbb{R}^n}
 \frac{\textbf{f}\;|\theta|^3\left(\sqrt{(\pi^1)^2-\theta^{-1}(\textbf{m}^2+
 \Theta_{AB}\pi^A\pi^B})-\pi^1\right)^2\sqrt{|\Theta|}}{\sqrt{(\pi^1)^2-\theta^{-1}(\textbf{m}^2+
 \Theta_{AB}\pi^A\pi^B})}
 d\pi.
\end{equation}
\end{proposition}
\subsection{The momentum constraints $\widetilde{X}_{1A}=0$ on $\mathcal{C}$}
 The Einstein equations
$\widetilde{X}_{1A}=0$ on $\mathcal{C}$ read
$\widetilde{R}_{1A}=\widetilde{T}_{1A}$. We are thus interested here
in the expression of $\widetilde{R}_{1A}$ resumed in (\ref{y23}) of
appendix \ref{A3} and the expression of $\widetilde{T}_{1A}$, which
in its non-explicit form is
\begin{equation*}
    \widetilde{T}_{1A}=\int_{\mathbb{R}^n}\textbf{f}\frac{(\widetilde{g}_{01})^2\widetilde{g}_{AB}\pi^0\pi^B}
    {\pi^0+\pi^1}\sqrt{|\widetilde{\gamma}|}\;d\pi .
\end{equation*}
 Combining this latter expression with the one of $\widetilde{R}_{1A}$ (\ref{y23}),
 using the expression of $\pi^0 $ (\ref{y24}) and
notations (\ref{y17}), the momentum
 constraints $\widetilde{X}_{1A}=0,\;A=2,...,n$ reduce on $\mathcal{C}$
 to partial differential relations in terms of the
 Cauchy data for $
\theta,\;\Theta=(\Theta_{AB}),\;\textbf{f},\;\psi_{1i},\;i=1,...,n$
for $(H_{\overline{g}},H_\rho)$.
 \begin{proposition}\label{p3}The hypotheses are
those of proposition \ref{p2}. Then the momentum constraints
$\widetilde{X}_{1A}=0$ on $\mathcal{C}$ are satisfied if and only if
the Cauchy data $
\theta,\;\Theta=(\Theta_{AB}),\;\textbf{f},\;\psi_{1i},\;i=1,...,n$
for $(H_{\overline{g}},H_\rho)$, agree with the partial differential
relations
\begin{equation*}
    \partial_1\psi_{1A}+(2^{-1}\Theta^{CB}\partial_1\Theta_{CB}-
    \theta^{-1}\partial_1\theta)\psi_{1A}+
\end{equation*}
\begin{equation*}
    \frac{1}{2}(\theta^{-1}\psi_{11}-
   \theta^{-1}\partial_1\theta- \frac{1}{2}\Theta^{CB}
\partial_1\Theta_{CB})\partial_A\theta
- \frac{\theta}{2}\partial_C(\Theta^{CB}
\partial_1\Theta_{AB})+
\end{equation*}
\begin{equation*}
 \frac{\theta}{2}\partial_A(\Theta^{CB}
\partial_1\Theta_{CB})   - \frac{\theta}{2}\partial_A\psi_{11}
    -          \frac{\theta}{4}(\Theta^{EF}\partial_C
          \Theta_{EF})(\Theta^{CB}\partial_1
          \Theta_{BA})+
\end{equation*}
\begin{equation*}
      \frac{\theta}{4}\Theta^{DE}\Theta^{CB}(\partial_1 \Theta_{DB})
       (\partial_A\Theta_{EC}+\partial_C\Theta_{EA}-
  \partial_E\Theta_{AC})
\end{equation*}
\begin{equation}\label{x19}
    =2 \int_{\mathbb{R}^n}
 \frac{\textbf{f}\;(\theta)^3\Theta_{AD}\left(-\pi^1+\sqrt{(\pi^1)^2-\theta^{-1}(\textbf{m}^2+
 \Theta_{CB}\pi^C\pi^B})\right)\pi^D\sqrt{|\Theta|}}{\sqrt{(\pi^1)^2-\theta^{-1}(\textbf{m}^2+
 \Theta_{CB}\pi^C\pi^B})}
 d\pi.
\end{equation}
\end{proposition}

\subsection{The momentum constraint $\widetilde{X}_{01}=0$}
The momentum constraint $\widetilde{X}_{01}=0$ is equivalent to
\begin{equation}\label{y26}
\widetilde{g}^{AB}\widetilde{R}_{AB}=\widetilde{g}^{01}(\widetilde{R}_{11}-2\widetilde{T}_{01}).
\end{equation}
 According to the expression of $\widetilde{R}_{AB}$
(\ref{y32}) in appendix \ref{A3}, one has:
\begin{equation*}
\widetilde{g}^{AB}\widetilde{R}_{AB}=\widetilde{g}^{01}\partial_1
(\widetilde{g}^{AB}\partial_0\widetilde{g}_{AB})+\frac{1}{2}\widetilde{g}^{01}
(\widetilde{g}^{AB}\partial_1\widetilde{g}_{AB}-\widetilde{g}^{01}\partial_0\widetilde{g}_{11})
(\widetilde{g}^{AB}\partial_0\widetilde{g}_{AB})+
\end{equation*}
\begin{equation*}
   +\frac{1}{2}\widetilde{g}^{01}(\widetilde{g}^{01}\partial_0\widetilde{g}_{11}
   -\frac{1}{2}\widetilde{g}^{CD}\partial_1\widetilde{g}_{CD}
   )
   (\widetilde{g}^{AB}\partial_1\widetilde{g}_{AB})+\widetilde{g}^{01}
   \widetilde{g}^{AB}\widetilde{\Gamma}^D_{AB}\partial_0\widetilde{g}_{1D}-
\end{equation*}
\begin{equation*}
   \widetilde{g}^{01}\widetilde{g}^{AB}\partial^2_{0A}\widetilde{g}_{1B}-
    \frac{1}{2}\widetilde{g}^{01}\widetilde{g}^{AB}\partial^2_{11}\widetilde{g}_{AB}-
    \widetilde{g}^{AB}\partial_C(\widetilde{\Gamma}^C_{AB})+
    \frac{1}{2}\widetilde{g}^{AB}\underset{(\gamma,\delta)\in
   \{(0,1);(1,0);(A,B)\}}{\underbrace{\partial_B[\widetilde{g}^{\gamma\delta}\partial_A
   \widetilde{g}_{\gamma\delta
   }}]}
\end{equation*}
\begin{equation*}
   + \frac{1}{2}(\widetilde{g}^{01})^2\widetilde{g}^{AB}
   (\partial_0\widetilde{g}_{1B}-\partial_B\widetilde{g}_{01})\partial_0\widetilde{g}_{1A}+
   \frac{1}{2}\widetilde{g}^{CD}\widetilde{g}^{AB}\partial_1\widetilde{g}_{CA}\partial_1\widetilde{g}_{DB}+
\end{equation*}
\begin{equation}\label{y25}
   \widetilde{g}^{AB}\widetilde{\Gamma}^D_{AC}\widetilde{\Gamma}^C_{DB}-
   \widetilde{g}^{AB}(\widetilde{g}^{01}\partial_C\widetilde{g}_{01}+\widetilde{\Gamma}^D_{DC})
   \widetilde{\Gamma}^C_{AB}.
\end{equation}
On the other hand $\widetilde{T}_{01}$ is given by:
\begin{equation}\label{y27}
    \widetilde{T}_{10}=\int_{\mathbb{R}^n}\textbf{f}\frac{|\widetilde{g}_{01}|^3\pi^0\pi^1}
    {\pi^0+\pi^1}\sqrt{|\widetilde{\gamma}|}\;d\pi .
\end{equation}Concerning the term $\widetilde{R}_{11}$ in the
right-hand side of equation (\ref{y26}), its expression is given in
(\ref{y22}). We remark however that in the scheme of resolution of
the constraints, the heavy term $\widetilde{R}_{11}$ is replaced  by
the expression of $\widetilde{T}_{11}$ as soon as the first
constraint is satisfied. Now setting $\chi=\Theta^{AB}\psi_{AB}$,
combining the relations (\ref{y25}), (\ref{y22}), (\ref{y27}), using
the expression of $\pi^0 $ (\ref{y24}), the Christoffel symbols
\begin{equation}\label{x33}
    \widetilde{\Gamma}_{AB}^C=\frac{1}{2}\Theta^{CD}(\partial_A\Theta_{BD}+\partial_{B}\Theta_{AD}-
\partial_D\Theta_{AB}),\;A,B,C,D=2,...,n,
\end{equation}
notations of (\ref{y17}), the momentum
 constraint $\widetilde{X}_{10}=0$ is a partial differential relation on $\mathcal{C}$
  in terms of the
 Cauchy data for $
\theta,\;\Theta=(\Theta_{AB}),\;\textbf{f},\;\psi_{1i},\;i=1,...,n,\;\chi=\Theta^{AB}\psi_{AB}$
for $(H_{\overline{g}},H_\rho)$.
 \begin{proposition}\label{p4}The hypotheses
are those of proposition \ref{p2}. Then the momentum constraint
$\widetilde{X}_{10}=0$ on $\mathcal{C}$ is satisfied if and only if
the Cauchy data $ \theta,\;\Theta=(\Theta_{AB}),\;\textbf{f}$,
$\psi_{1i},\;i=1,...,n,\chi=\Theta^{AB}\psi_{AB}$ for
$(H_{\overline{g}},H_\rho)$, satisfy the partial differential
relation
\begin{equation*}
\partial_1\chi
+2^{-1}(\Theta^{AB}\partial_1\Theta_{AB}-\theta^{-1}\psi_{11})\chi
 +\frac{1}{2}(\theta^{-1}\psi_{11}
   -\frac{1}{2}\Theta^{CD}\partial_1\Theta_{CD}
   )
   (\Theta^{AB}\partial_1\Theta_{AB})+
\end{equation*}
\begin{equation*}
   \Theta^{AB}(\widetilde{\Gamma}^D_{AB}\psi_{1D}-
   \partial_{A}\psi_{1B}-
    \frac{1}{2}\partial^2_{11}\Theta_{AB}-
\theta   \partial_C(\widetilde{\Gamma}^C_{AB})+ \theta
\partial_B[\theta^{-1}\partial_A
   \theta
   ])+
\end{equation*}
\begin{equation*}
   \frac{\Theta^{AB}}{2}(
    \theta\partial_B[\Theta^{CD}\partial_A
   \Theta_{CD}]+
\theta^{-1}(\psi_{1B}-\partial_B\theta)\psi_{1A}+
   \theta\Theta^{CD}\partial_1\Theta_{CA}\partial_1\Theta_{DB})+
\end{equation*}
\begin{equation*}
\theta
\Theta^{AB}(\widetilde{\Gamma}^D_{AC}\widetilde{\Gamma}^C_{DB}-
      (\theta^{-1}\partial_C\theta+\widetilde{\Gamma}^D_{DC})
   \widetilde{\Gamma}^C_{AB})=\widetilde{R}_{11}-
\end{equation*}
\begin{equation}\label{x20}
   2\int_{\mathbb{R}^n}\textbf{f}\frac{|\theta|^3
    \left(\sqrt{(\pi^1)^2-\theta^{-1}(\textbf{m}^2+
 \Theta_{AB}\pi^A\pi^B})-\pi^1\right)\pi^1}
    {\sqrt{(\pi^1)^2-\theta^{-1}(\textbf{m}^2+
 \Theta_{AB}\pi^A\pi^B)}}\sqrt{|\Theta|}\;d\pi.
\end{equation}
\end{proposition}
\subsection{The constraints $\widetilde{X}_{AB}-\frac{\widetilde{g}^{CD}
\widetilde{X}_{CD}}{n-1}\widetilde{g}_{AB}=0$} For the constraints
$\widetilde{X}_{AB}-\frac{\widetilde{g}^{CD}\widetilde{X}_{CD}}{n-1}\widetilde{g}_{AB}=0$,
they are equivalent to
\begin{equation}\label{y28}
\widetilde{R}_{AB}-\widetilde{T}_{AB}-\frac{\widetilde{g}^{CD}
(\widetilde{R}_{CD}-\widetilde{T}_{CD})}{n-1}\widetilde{g}_{AB}=0.
\end{equation}
 One needs the expression of $\widetilde{R}_{AB}$ (\ref{y32}) in
appendix \ref{A3}, and the one of $\widetilde{T}_{AB}$ given by
\begin{equation}\label{y29}
\widetilde{T}_{AB}=-\int_{\mathbb{R}^n}\textbf{f}\frac{|\widetilde{g}_{01}
|\widetilde{g}_{AC}\widetilde{g}_{BD}\pi^C\pi^D}{\pi^0+\pi^1}\sqrt{|\widetilde{\gamma}|}\;d\pi.
\end{equation}We signal that in the hierarchy of resolution of
constraints, the quantity $\widetilde{g}^{CD}\widetilde{R}_{CD}$ in
the right-hand side of (\ref{y28}) is substituted by
$\widetilde{g}^{01}(\widetilde{T}_{11}-\widetilde{T}_{01})$ as soon
as the previous constraints are satisfied. All that precedes implies
that the constraints
$\widetilde{X}_{AB}-\frac{\widetilde{g}^{CD}\widetilde{X}_{CD}}{n-1}\widetilde{g}_{AB}=0$
are partial differential in terms of the Cauchy data $
\theta,\;\Theta=(\Theta_{AB}),\;\textbf{f}$,
$\psi_{1i},\;i=1,...,n,\psi_{AB},\; A,B=2,...,n$ for
$(H_{\overline{g}},H_\rho)$.
 \begin{proposition}\label{p5}The hypotheses
are those of proposition \ref{p2}. Then the constraints
$\widetilde{X}_{AB}-\frac{\widetilde{g}^{CD}\widetilde{X}_{CD}}{n-1}\widetilde{g}_{AB}=0$
on $\mathcal{C}$ are satisfied if and only if the Cauchy data $
\theta,\;\Theta=(\Theta_{AB}),\;\textbf{f}$,
$\psi_{1i},\;i=1,...,n,\psi_{AB},\; A,B=2,...,n$ for
$(H_{\overline{g}},H_\rho)$, satisfy the partial differential
relations
\begin{equation*}
    \partial_1\psi_{AB}-
\end{equation*}
\begin{eqnarray*}
    \frac{1}{2}(
    \frac{1}{2}(\theta^{-1}\psi_{11}-\Theta^{EF}\partial_1\Theta_{EF})\delta_A^C\delta_B^D+
    \Theta^{ED}\partial_1\Theta_{EB}\delta_A^C+\Theta^{ED}\partial_1\Theta_{DA}\delta_B^C
    )\psi_{CD}+
\end{eqnarray*}
\begin{equation*}
    \frac{1}{2}(\frac{1}{2}\Theta^{CE}\psi_{CE}
+(\theta^{-1}\psi_{11}
   -\frac{1}{2}\Theta^{CD}\partial_1\Theta_{CD}
   ))\partial_1\Theta_{AB}+
\end{equation*}
\begin{equation*}
   \widetilde{\Gamma}^D_{AB}\psi_{1D}-
   \partial_{A}\psi_{1B}-\partial_{B}\psi_{1A}-
    \frac{1}{2}\partial^2_{11}\Theta_{AB}-
\theta   \partial_C(\widetilde{\Gamma}^C_{AB})+ \theta
\partial_B(\theta^{-1}\partial_A
   \theta
   )+
\end{equation*}
\begin{equation*}
   \frac{1}{2}(
    \theta\partial_B(\Theta^{CD}\partial_A
   \Theta_{CD})+
\theta^{-1}(\psi_{1B}-\partial_B\theta)\psi_{1A}+
   \Theta^{CD}\partial_1\Theta_{CA}\partial_1\Theta_{DB})+
\end{equation*}
\begin{equation*}
  \widetilde{\Gamma}^D_{AC}\widetilde{\Gamma}^C_{DB}-
      (\theta^{-1}\partial_C\theta+\widetilde{\Gamma}^D_{DC})
   \widetilde{\Gamma}^C_{AB}=  \frac{\Theta^{CD}
(\widetilde{R}_{CD}-\widetilde{T}_{CB})}{n-1}\Theta_{AB}-
\end{equation*}
\begin{equation}\label{x21}
    \int_{\mathbb{R}^n}\textbf{f}\frac{|\theta
|\Theta_{AC}\Theta_{BD}\pi^C\pi^D}{\sqrt{(\pi^1)^2-\theta^{-1}(\textbf{m}^2+
 \Theta_{CD}\pi^C\pi^D)}}\sqrt{|\Theta|}\;d\pi.
\end{equation}
\end{proposition}
\begin{remark}
The constraints of the type in proposition \ref{p5} appear in the
case of the "Double null foliation gauge" \cite{a},\cite{h}.
\end{remark}
\subsection{The constraint $\widetilde{
X}_{00}-\widetilde{g}_{01}\frac{\widetilde{g}^{CD}\widetilde{X}_{CD}}{n-1}+
   \widetilde{g}_{01}\frac{\partial (\widetilde{\Gamma}^0+\widetilde{\Gamma}^1)}{\partial
   y^0}   =0$}
The last constraint in our scheme corresponds to
\begin{equation*}
    \widetilde{
X}_{00}-\widetilde{g}_{01}\frac{\widetilde{g}^{CD}\widetilde{X}_{CD}}{n-1}+
   \widetilde{g}_{01}\frac{\partial (\widetilde{\Gamma}^0+\widetilde{\Gamma}^1)}{\partial
   y^0}   =0,\;A,B,C,D=2,...,n;
\end{equation*}
 it is equivalent to
   \begin{equation}\label{y31}
\widetilde{R}_{00}-\widetilde{T}_{00}-
   \widetilde{g}_{01}\frac{\widetilde{g}^{AB}(\widetilde{R}_{AB}-\widetilde{T}_{AB})}{n-1}+
   \widetilde{g}_{01}\frac{\partial (\widetilde{\Gamma}^0+\widetilde{\Gamma}^1)}{\partial
   y^0}   =0.
   \end{equation}
    One needs at this level the expression of $\widetilde{R}_{00}$ (\ref{y36}) in appendix
    \ref{A3}, the one of
   $\frac{\partial (\widetilde{\Gamma}^0+\widetilde{\Gamma}^1)}{\partial
   y^0} $ (\ref{y35}) in  appendix \ref{A4}, and
   $\widetilde{T}_{00}$ which resumes to
   \begin{equation}\label{y30}
        \widetilde{T}_{00}=\int_{\mathbb{R}^n}\textbf{f}|\widetilde{g}_{01}|^3(\pi^0+\pi^1)
        \sqrt{|\widetilde{\gamma}|}\;d\pi.
   \end{equation}
\begin{proposition}\label{p6}
The hypotheses are those of proposition \ref{p2}. Then the
constraint $\widetilde{
X}_{00}-\widetilde{g}_{01}\frac{\widetilde{g}^{CD}\widetilde{X}_{CD}}{n-1}+
   \widetilde{g}_{01}\frac{\partial (\widetilde{\Gamma}^0+\widetilde{\Gamma}^1)}{\partial
   y^0}   =0;\;C,D=2,...,n$ on $\mathcal{C}$ is satisfied
if and only if the Cauchy data $
\theta,\;\Theta=(\Theta_{AB}),\;\textbf{f}$,
$\psi_{1\alpha},\;\alpha=0,1,...,n$, $\psi_{AB},\; A,B=2,...,n$ for
$(H_{\overline{g}},H_\rho)$, satisfy the partial differential
relation
\begin{equation*}
\partial_1\psi_{01}+(\frac{1}{4}\Theta^{AB}\partial_1\Theta_{AB}-\theta^{-1}(\partial_1\theta
        +\psi_{11}))\psi_{01}-\frac{1}{2}\Theta^{AD}\psi_{1D}\psi_{1A}+
\end{equation*}
\begin{equation*}
    \frac{\theta^{-1}}{4}\psi_{11}^2+\frac{\theta}{4}\Theta^{DB}\Theta^{CE}
    \psi_{BC}\psi_{ED}+\frac{\theta^{-1}}{4}\psi_{11}\partial_1\theta-\frac{1}{2}\Theta^{CB}\psi_{1C}(\partial_B\theta)
\end{equation*}
\begin{equation*}
  +\frac{\theta^{-1}}{2}\psi_{01}^2 + \frac{\theta}{2}\partial_A\Theta^{AC}\partial_C\theta-\frac{1}{2}\partial^2_{11}\theta+
    \frac{\theta}{2}\Theta^{AB}\partial^2_{AB}\theta+\frac{1}{4}\partial_1\theta
    \Theta^{CB}(\psi_{CB}-\partial_1\Theta_{CB})
\end{equation*}
   \begin{equation*}
 +\frac{\theta}{2}\Theta^{DB}(\partial_B\theta)\widetilde{\Gamma}^C_{CD}
        =
   \frac{\Theta^{AB}(\widetilde{R}_{AB}-\widetilde{T}_{AB})}{n-1}+
   \end{equation*}
   \begin{equation}\label{x22}
        \theta\int_{\mathbb{R}^n}\textbf{f}|\theta|^3(\sqrt{(\pi^1)^2-\theta^{-1}(\textbf{m}^2+
 \Theta_{AB}\pi^A\pi^B)})
        \sqrt{|\Theta|}\;d\pi.
   \end{equation}
\end{proposition}
  \section{On the resolution of the initial data constraints }
     The cone $\mathcal{C}$ admits the equation $x^0-\sqrt{\sum_i(x^i)^2}=0$
in coordinates $(x^\mu)$ and this is viewed as the requirement
 that the coordinates $(x^\delta)$ coincide on $\mathcal{C}$ with some normal coordinates
 based at O and attached to an appropriate basis of vectors at O.
  The Christoffel symbols of the metric thus vanish at O. This
 induces that a regular metric in a neighborhood of the cone and
Minkowskian at O must satisfy the expansion
\begin{equation}\label{x15}
g=g_{\mu\nu}dx^\mu dx^\nu=
        (\eta_{\mu\nu}+O(r^2))dx^\mu dx^\nu.
 \end{equation}
  Its behavior at O (see also \cite{l} for some details) in coordinates $(y^\mu)$ is thus:
  \begin{equation*}
    \widetilde{g}_{01}=-1+O(r^2),\;\widetilde{g}_{AB}=r^2\sum_i\frac{\partial
\theta^i}{\partial y^A}
    \frac{\partial \theta^i}{\partial y^B}+O(r^4),
  \end{equation*}
 \begin{equation}\label{x16}
  \widetilde{g}^{AB}=
    \sum_s\frac{\partial y^A}{\partial x^s}\frac{\partial y^B}{\partial
    x^s}+O(cste),
 \end{equation}where the expression of $\frac{\partial y^A}{\partial x^s}$
 is given by:
\begin{equation}
    \frac{\partial y^A}{\partial x^s}=\frac{1}{y^1}\left(\frac{\partial y^A}{\partial \theta^s}-
    \frac{\partial y^A}{\partial \theta^k}\theta^k \theta^s\right),
\end{equation}
 moreover, one has the following properties:
 \begin{equation*}
    {[\partial_0 \widetilde{g}_{00}]}=O(r),\;
    {[\partial_0\widetilde{g}_{01}]}=O(r),\;
    {[\partial_0 \widetilde{g}_{0A}]}=0,\;[\partial_0 \widetilde{g}_{1A}]=O(r^2),
 \end{equation*}
 \begin{equation}\label{c10}
       {[\partial_0 \widetilde{g}_{11}]}=O(r),\;
{[\partial_0 \widetilde{g}_{AB}]}=O(r^3).
 \end{equation}
 \subsection{The free data and their many ways}
 The constraints (\ref{x18}),(\ref{x19}),(\ref{x20}),(\ref{x21}),(\ref{x22}) do not belong to a specific
 type of differential equations which can be solved using
  classical methods unless free data are well-chosen.
 Indeed, there are some free data which guarantees that
 the constraints (\ref{x18}),(\ref{x19}),(\ref{x20}),(\ref{x21}),(\ref{x22})
  can be solved
 (hierarchically) algebraically or/and as propagations equations along
 null generators of the cone $\mathcal{C}$ in terms of the constrained Cauchy
 data. Their specificity rests on how to tackle the first constraint. To use a terminology
 which appears in other contexts \cite{a} and thereby highlight the many ways of
 the free data,
  we introduce for the tangent vectors $X,Y\in TC_{\mu\nu}$ the one form $\xi$
  and the second fundamentals forms
  $\chi,\;\underline{\chi}$ of $C_{\mu\nu}:=\{y^0=\mu,\;y^1=\nu\}$, defined by
  \begin{equation}\label{x30}
\xi (X)=\frac{1}{2}g(\nabla_X ^{\partial_1},\partial_0),
  \end{equation}
  \begin{equation}\label{x29}
    \chi (X,Y)=g(\nabla^{\partial_1}_X,Y)
    ,\;\underline{\chi }(X,Y)=g(\nabla^{\partial_0}_X,Y),
  \end{equation}and correspondingly the shear tensors
  $\widehat{\chi},\;\widehat{\underline{\chi}}$:
  \begin{equation}\label{x31}
\widehat{\chi}=\chi-\frac{tr\chi}{n-1}\gamma,\;\widehat{\underline{\chi}}=\underline{\chi}-
\frac{tr\underline{\chi}}{n-1}\gamma,
  \end{equation}where $tr\chi$ and $tr\underline{\chi}$ denote
  the trace of $\chi$ and $\underline{\chi}$ respectively, with respect to
  the induced metric $\gamma$ of $g$ on $C_{\mu\nu}$. On $\mathcal{C}=\cup
  C_{o\nu}$, these geometric objects are linked to the gravitational
  Cauchy data, ie.:
  \begin{equation}
    \xi_A=\frac{1}{4}(\partial_A\widetilde{g}_{01}-\partial_0\widetilde{g}_{1A}),\;
    \chi_{AB}=\frac{1}{2}\partial_1\widetilde{g}_{AB},\;
    \underline{\chi}_{AB}=\frac{1}{2}\partial_0\widetilde{g}_{AB}.
  \end{equation}The constraints (\ref{x18}),(\ref{x19}),(\ref{x20}),(\ref{x21}),(\ref{x22})
   expressed
  then in terms of these geometric objects and other quantities identified as
   combinations or not of the Christoffel symbols of the metric
  on $\mathcal{C}$, as instance, one has:
  \begin{equation*}
    2g(\nabla_{\partial_0}^{\partial_1},\partial_1)=\partial_0\widetilde{g}_{11},\;
2g(\nabla_{\partial_0}^{\partial_0},\partial_0)=\partial_0\widetilde{g}_{01},\;
2g(\nabla_{\partial_0}^{\partial_0},\partial_1)=2\partial_0\widetilde{g}_{01}-
\widetilde{g}^{01}\partial_1\widetilde{g}_{01}.
  \end{equation*}Here we make discussions according to the affine
  parametrization condition.
\subsubsection{No initial condition on the non-affinity constant
$\kappa:=\widetilde{\Gamma}^1_{11}$}
 The assumption (A) (\ref{x3})
guarantees only that the vector fields $\frac{\partial}{\partial
y^1}$ is tangent to the null generators of the cone. According to
the notations of (\ref{y17}), one has
$\nabla_{\partial_1}^{\partial_1}=\kappa
\partial_1=\frac{1}{2\theta}(2\partial_1\theta-\psi_{11})\partial_1$.
It appears thus clearly that $\kappa$ is directly linked to the
Cauchy data $\psi_{11}$. If no condition is given on $\kappa$, then
the first constraint equation is solvable algebraically in term of
$\psi_{11}$ provided all the components of the metric
$\theta,\;\Theta_{AB}$ are given, with the divergence scalar
$tr\chi=\frac{1}{2}\Theta^{AB}\partial_1 \Theta_{AB}$ of
$\mathcal{C}$ which is nowhere vanishing. In this case, the solution
is global.
\begin{remark}\label{r1}
Given a metric
 $\widetilde{\gamma}=(\widetilde{\gamma}_{AB})$ on $\mathcal{C}$, a conformal
 class
 of metric $(e^\omega\widetilde{\gamma}_{AB})$
   satisfying the "nowhere vanishing condition for the divergence scalar"
     realizes for $\omega$ given by
    \begin{equation}
        \omega(y^i)=-\frac{1}{n-1}\int_{0}^{y_1}(\widetilde{\gamma}^{AB}
        \partial_1\widetilde{\gamma}_{AB}-\omega_0)(\lambda,y^A)d\lambda,
    \end{equation}where $\omega_0\equiv \omega_0(y^1,y^A)$ is any function s.t.
    $|\omega_0|>0,\; y^1\neq 0$.
\end{remark}
\subsubsection{$\kappa$ and $(\widetilde{\gamma}_{AB})$ as free data}
Given $\kappa $ as free data is equivalent to the given of the
Cauchy data $\psi_{11}$ as soon as $\theta <0$ is known, and
indicates the specification of an affine parameter though its
expression is not necessarily trivial. The first constraint equation
is solvable in term of a conformal factor $\omega$ s.t.
$\Theta_{AB}=e^\omega \widetilde{\gamma}_{AB}$ as in \cite{j}, where
the $\widetilde{\gamma}_{AB}$ are prescribed freely and make up a
symmetric positive definite matrix. Indeed,
\begin{equation*}
  \Theta^{CB}\partial_1
\Theta_{CB} =(n-1) \partial_1\omega+
\widetilde{\gamma}^{BC}\partial_1\widetilde{\gamma}_{BC},
\end{equation*}
\begin{equation*}
  \partial_1\left( \Theta^{CB}\partial_1
\Theta_{CB} \right) = (n-1)\partial_1(\partial_1\omega)+
\partial_1\left(\widetilde{\gamma}^{BC}\partial_1\widetilde{\gamma}_{BC}\right),
\end{equation*}
\begin{equation*}
    (\Theta^{CB}\partial_1\Theta_{BD})
    (\Theta^{DE}\partial_1\Theta_{CE}) =(n-1) \left[\partial_1\omega\right]^2+2\left[\widetilde{\gamma}^{EC}\partial_1
\widetilde{\gamma}_{EC}\right][\partial_1\omega]+
\end{equation*}
\begin{equation*}
 \widetilde{\gamma}^{BC}\widetilde{\gamma}^{DE}
(\partial_1\widetilde{\gamma}_{CE})(\partial_1\widetilde{\gamma}_{BD}).
\end{equation*}
The first constraint equation then reads:
\begin{equation*}
    \partial^2_1\omega+\left(-\kappa+
    \frac{\widetilde{\gamma}^{EC}\partial_1\widetilde{\gamma}_{EC}}{n-1}\right)\partial_1\omega
    +\frac{(\partial_1\omega)^2}{2}+\frac{\partial_1(\widetilde{\gamma}^{EC}\partial_1\widetilde{\gamma}_{EC})}{n-1}
\end{equation*}
\begin{equation*}
   - \frac{\kappa}{n-1}\widetilde{\gamma}^{BC}\partial_1\widetilde{\gamma}_{BC}+
    \frac{\widetilde{\gamma}^{BC}\widetilde{\gamma}^{DE}
    \partial_1\widetilde{\gamma}_{CE}\partial_1\widetilde{\gamma}_{BD}}{2(n-1)}
\end{equation*}
\begin{equation*}
  =  \frac{2}{1-n}\int_{\mathbb{R}^n}
 \frac{\textbf{f}\;|\theta|^3\left(\sqrt{(\pi^1)^2-\theta^{-1}(\textbf{m}^2+
 e^\omega\widetilde{\gamma}_{AB}\pi^A\pi^B})-\pi^1\right)^2
 e^{\frac{n-1}{2}\omega}\sqrt{|\widetilde{\gamma}|}}{\sqrt{(\pi^1)^2-\theta^{-1}(\textbf{m}^2+
 e^\omega\widetilde{\gamma}_{AB}\pi^A\pi^B})}
 d\pi.
\end{equation*}
In term of $\Omega:=e^{\frac{\omega}{2}}$, one has:
\begin{equation*}
    2\partial^2_1\Omega+2\left(-\kappa+
    \frac{\widetilde{\gamma}^{EC}\partial_1\widetilde{\gamma}_{EC}}{n-1}\right)\partial_1\Omega+
    \end{equation*}
\begin{equation*}
  \left(\frac{\partial_1(\widetilde{\gamma}^{EC}\partial_1\widetilde{\gamma}_{EC})}{n-1}
   - \frac{\kappa}{n-1}\widetilde{\gamma}^{BC}\partial_1\widetilde{\gamma}_{BC}+
    \frac{\widetilde{\gamma}^{BC}\widetilde{\gamma}^{DE}
    \partial_1\widetilde{\gamma}_{CE}\partial_1\widetilde{\gamma}_{BD}}{2(n-1)}\right)\Omega
\end{equation*}
\begin{equation}\label{x34}
  =  \frac{2\Omega^n}{1-n}\int_{\mathbb{R}^n}
 \frac{\textbf{f}\;|\theta|^3\left(\sqrt{(\pi^1)^2-\theta^{-1}(\textbf{m}^2+
 \Omega^2\widetilde{\gamma}_{AB}\pi^A\pi^B})-\pi^1\right)^2
 \sqrt{|\widetilde{\gamma}|}}{\sqrt{(\pi^1)^2-\theta^{-1}(\textbf{m}^2+
 \Omega^2\widetilde{\gamma}_{AB}\pi^A\pi^B})}
 d\pi.
\end{equation}One can also use the following relation to simplify some expressions above:
\begin{equation*}
\widetilde{\gamma}^{EC}\partial_1\widetilde{\gamma}_{EC}=2\partial_1\ln\sqrt{|\widetilde{\gamma}|}.
\end{equation*}
This second order differential equation has a unique solution given
Minkowskian initial values at O.
\subsubsection{$\kappa$ and components $\sigma_A^B$ of the shear w.r.t. dual bases as free data}
If the components $\sigma_A^B$ of the shear w.r.t. dual bases
$\{\frac{\partial}{\partial y^A}\},\;\{dy^A\}$ are given as free
data together with $\theta <0$ and $\kappa$, the relation
$\chi_{AB}=\frac{1}{2}\partial_1\Theta_{AB}$ and the first
constraint equation induce a system of differential equations, of
unknowns the $\Theta_{AB}$ and the divergence scalar $tr\chi$.
Indeed, one sets: $\widehat{\chi}^B_A=\sigma^B_A$, and one can
deduce:
\begin{equation}\label{x35}
  \partial_1\Theta_{AB}-2(\sigma_B^C+\frac{tr\chi}{n-1}\delta_B^C)\Theta_{AC} = 0,
\end{equation}
\begin{equation}\label{x36}
  \partial_1tr\chi-\frac{\theta^{-1}}{2}(\theta\kappa+\partial_1\theta)tr\chi + \frac{(tr\chi)^2}{n-1}+
  \sigma_A^B\sigma_B^A
  -\widetilde{T}_{11}(\theta,\Theta_{CD},\textbf{m}^2,\textbf{f})=0.
\end{equation} This system has a unique solution given
Minkowskian initial values at O.  Setting
\begin{equation*}
    V=e^{\int \frac{tr\chi}{n-1}dy^1},
\end{equation*}the equation (\ref{x36}) reads:
\begin{equation}\label{x37}
    \partial_{11}^2V-\frac{\theta^{-1}}{2}(\theta\kappa+\partial_1\theta)\partial_1
    V+\frac{(-\sigma_A^B\sigma_B^A
  +\widetilde{T}_{11}(\theta,\Theta_{CD},\textbf{m}^2,\textbf{f}))}{1-n}V=0.
\end{equation}
\begin{remark}[Shear prescribed ]
One can prescribe rather the components $\sigma_{AB}$ of the shear
with respect to the basis $\{dy^A\}$, however the system
(\ref{x35}), (\ref{x36}) seems simpler than the one obtained in this
case.
\end{remark}
\begin{remark}[On global solution in vacuum for prescribed $\kappa$]
From the analysis above, it appears that in vacuum and for
prescribed $\kappa$, one has a global solution. Indeed, in this
case, the equation (\ref{x34}) is a homogeneous linear second order
differential equation in $\Omega$, and the same property holds for
the equation (\ref{x37}) of unknown $V$ which can be solved
independently of the equation (\ref{x35}).
\end{remark}
\subsection{Constraints's solutions: the general idea}
 In this paper, we concentrate on the first case where no condition is given on $\kappa$,
  this corresponds thus to unconstrained initial metric and the solutions of
  the constraints yield the first fundamental form of the initial
  hypersurface.
 The prescribed free data comprise precisely:\\
  (\textbf{a})- $\mathcal{C}^\infty$ functions
$\widetilde{\gamma}_{AB}\equiv\widetilde{\gamma}_{AB}(y^i)$ that
make up (for $y^1\neq 0$) a symmetric positive definite matrix, and
satisfy the "nowhere vanishing" property for the divergence scalar,
ie.:
\begin{equation}\label{x23}
    \left|\widetilde{\gamma}^{AB}\frac{\partial
\widetilde{\gamma}_{AB}}{\partial y^1}\right|>0,\;y^1\neq 0.
\end{equation}
The existence of a large class of such free data
$(\widetilde{\gamma}_{AB})$ is analyzed in remark \ref{r1}.\\
 (\textbf{b})- a smooth function
$\theta\equiv\theta (y^i) $ on $\mathcal{C}$, and $\textbf{f}\equiv
\textbf{f}(y^i,\pi^j)$ on $\mathcal{C}\times \mathbb{R}^n$ s.t.
$\theta$ is negative, $\textbf{f}$ is non negative of compact
support contained in $\{\pi^1>c_1>0\}$ for a mass $\textbf{m}\neq
0$; and for a zero mass the support of $\textbf{f}$ is contained in
$\{\pi^1>c_1>0,\;\sum_{A=2}^n(\pi^A)^2>c_2>0\}$, besides that, $
Supp(\textbf{f})\cap (\{O\}\times \mathbb{R}^n)=\emptyset$. These
free data satisfy
\begin{equation}\label{x61}
    \theta=-1+O(r^2),\;\widetilde{\gamma}_{AB}=r^2\sum_i\frac{\partial
\theta^i}{\partial y^A}
    \frac{\partial \theta^i}{\partial y^B}+O(r^4),
\end{equation}
\begin{equation}\label{x24}
  \widetilde{\gamma}^{AB}=
    \sum_s\frac{\partial y^A}{\partial x^s}\frac{\partial y^B}{\partial x^s}+O(cste).
 \end{equation}
  \begin{theorem}\label{th2}
  Given the free data as described above by
(\textbf{a})-(\textbf{b}). Then, there exists a unique global
solution $(\theta,\;\Theta_{AB}
,\;\psi_{1\nu},\;\psi_{AB},\;\textbf{f})$ on $\mathcal{C}\times
\mathbb{R}^n$ of the initial data constraints (\ref{x4}),
 (\ref{y7}), (\ref{x5}) for the
Einstein-Vlasov system.
\end{theorem}
 \begin{proof} Given the free data
(\textbf{a})-(\textbf{b}), one solves the constraints described by
(\ref{x18}), (\ref{x20}), (\ref{x21}), (\ref{x22}) in a hierarchical
scheme. Indeed, one sets: \\
$\Theta_{AB}(y^1,y^A)=\widetilde{\gamma}_{AB}(y^1,y^A)$, then
$|\Theta^{AB}\partial_1\Theta_{AB}|>0$, and $\psi_{11}$ solves
algebraically the Hamiltonian constraint $\widetilde{X}_{11}=0$ as
described by (\ref{x18}), with $\psi_{11}=O(y^1)$. For the
constraints $\widetilde{X}_{1A}=0$ (\ref{x19}) and
$\widetilde{X}_{01}=0$ (\ref{x20}), they can be written in the forms
\begin{equation}\label{x25}
    \frac{d\psi_{1A}}{dy^1}+\frac{(n-1)}{y^1}\psi_{1A}+\psi_{A}(y^i,\psi_{1C})=0,
\end{equation}
\begin{equation}\label{x28}
    \frac{d\chi}{dy^1}+\frac{(n-1)}{y^1}\chi+\psi(y^i,\chi)=0,
\end{equation}where the functions $\psi_A$ and $\psi$ are linear
w.r.t. $\psi_{1C}$ and respectively $\chi$, their solutions satisfy
the integral systems
\begin{equation}\label{x59}
    \psi_{1A}=\frac{1}{y^1}\int_{0}^{y^1}[- \lambda
\psi_{A}+(2-n)\psi_{1A}](\lambda,y^A)d\lambda,
\end{equation}
\begin{equation}\label{x60}
\chi=\frac{1}{y^1}\int_{0}^{y^1}[- \lambda
\psi+(2-n)\chi](\lambda,y^A)d\lambda.
\end{equation}
Since the functions under the integral's
 sign are continuous w.r.t. $y^1$ and
 Lipschitzian w.r.t. the corresponding unknowns, the solutions of the systems (\ref{x25}),
  (\ref{x28}) exist,
  are unique and global thanks to linearity, furthermore
  $\psi_{1A}=O((y^1)^2),\;\chi=O(y^1)$ according to the behavior near O of $\psi_A$ and $\psi$.
   One first solves the system
  in $\psi_{1A}$ and, after that, the equation regarding $\chi$.
  Now, we consider the constraints
$Z_{AB}\equiv \widetilde{X}_{AB}-\frac{\widetilde{g}^{CD}
\widetilde{X}_{CD}}{n-1}\widetilde{g}_{AB}=0$ described in
(\ref{x21}) for $(A,B)\neq (2,2)$ since $(Z_{AB}),\;(A,B=2,...,n)$
is a traceless tensor, of unknowns $\psi_{AB}$ for $(A,B)\neq
    (2,2)$ provided $\psi_{22}$ takes the value $
\psi_{22}=\frac{1}{\Theta^{22}}(\chi-\displaystyle\sum_{(A,B)\neq
    (2,2)}\Theta^{AB}\psi_{AB})$, and where $\Theta^{AB}\widetilde{R}_{AB}\equiv
\widetilde{R}^{(n-1)}$ equals
$\frac{(\widetilde{T}_{11}-2\widetilde{T}_{01})}{\theta}$ since
$\widetilde{X}_{01}=0$ is satisfied. This system is of the form
\begin{equation}\label{x26}
    \frac{d\psi_{AB}}{dy^1}-\frac{2}{y^1}\psi_{AB}+
    \psi'_{AB}(y^i,\psi_{BC})=0,
\end{equation}its solution is unique, global, satisfies the integral system
\begin{equation*}
    \psi_{AB}=\frac{1}{y^1}\int_{0}^{y^1}[- \lambda
\psi'_{AB}+3\psi_{AB}](\lambda,y^A)d\lambda,
\end{equation*} and by closer inspection of the expression of
$\psi'_{AB}$, one proves that $\psi_{AB}=O((y^1)^3)$.
 That $Z_{22}=0$ is also satisfied with $
\psi_{22}=\frac{1}{\Theta^{22}}(\chi-\displaystyle\sum_{(A,B)\neq
    (2,2)}\Theta^{AB}\psi_{AB})$ follows
from the traceless property of $(Z_{AB})$. One also has
$\psi_{22}=O((y^1)^3)$. The last constraint (\ref{x22}) has the form
\begin{equation}\label{x27}
    \frac{d\psi_{01}}{dy^1}+\frac{n-1}{2y^1}\psi_{01}+
    \psi'_{01}(y^i,\psi_{01})=0,
\end{equation}
its solution is unique, global, agrees with the integral equation
\begin{equation*}
    \psi_{01}=\frac{1}{y^1}\int_{0}^{y^1}[- \lambda
\psi'_{01}+\frac{(3-n)}{2}\psi_{01}](\lambda,y^A)d\lambda.
\end{equation*}At least by studying the behavior near O of
 the expression of $\psi'$ one deduces that $\psi_{01}=O(y^1).$
\end{proof}
\section{On the evolution system $(H_{\overline{g}},H_\rho)$}
The treatment of the evolution system requires more analysis of the
free data, namely a complete description of their behavior near $O$,
and obviously the behavior near the vertex O of the constraints's
solutions. In order to apply Dossa's well posedness results
\cite{o}, it would be interesting in a subsequent work, to prove
that the characteristic initial data on concerned here can arise as
restrictions to the cone of functions smooth at the neighborhood of
the tip of the cone. We remark that for this purpose, the lapse is
related to the metric $\overline{g}$ by the relation
$\tau=\frac{|\frac{D(x)}{D(y)}|}{\sqrt{|\widetilde{\gamma}|}}\sqrt{|\overline{g}|}$.
Another more technical issue would be to derive the above results
under lower regularity assumptions. Moreover, the construction of a
large class of initial data sets offers also here the possibility to
study without any symmetry assumptions the global future evolution
of small data with appropriate fall-off behavior at infinity. The
strong nonlinear features of the Einstein equations requires one to
rely on a quite rigid analytic approach based on energy estimates
and other many tools, for this, one should take advantage of results
on global existence of D. Christodoulou and S. Klainerman \cite{r},
S. Klainerman and F. Nicolo \cite{y}, H. Linblad and I. Rodniansky
\cite{z},  Y. Choquet-Bruhat \cite{m}, L. Bieri and N. Zipser
\cite{p}, P. G. LeFloch and Y. Ma \cite{e}.

\begin{appendices}
\section {Some algebraic relations including the gravitational data
on $\mathcal{C}$}\label{A1} From the zero-shift condition and
algebraic properties between the metric and its inverse, one shows
that the gravitational data satisfy the following relations on
$\mathcal{C}$.
\begin{equation*}
    \partial_0\widetilde{g}_{0A}=0,\;
 \partial_0 \widetilde{g}^{01}=-(\widetilde{g}^{01})^2
    (\partial_0
    \widetilde{g}_{01}-\partial_0\widetilde{g}_{11})=
-\partial_0 \widetilde{g}^{11},
\end{equation*}
\begin{equation*}
 \partial_0
    \widetilde{g}^{00}=-(\widetilde{g}^{01})^2\partial_0\widetilde{g}_{11},
       \partial_0
\widetilde{g}^{0C}=-\widetilde{g}^{01}\widetilde{g}^{CD}\partial_0
\widetilde{g}_{1D},\;\partial_0\widetilde{g}^{AD}=-\widetilde{g}^{AB}\widetilde{g}^{CD}\partial_0
\widetilde{g}_{CB},
\end{equation*}
\begin{equation}\label{m31}
\partial_0\widetilde{g}_{00}=\partial_0\widetilde{g}_{01},
\frac{\partial}{\partial
y^\nu}(\widetilde{g}^{0i}+\widetilde{g}^{1i})=0,\;\nu=0,...,n,\;i=1,...,n.
\end{equation}

\section{Christoffel symbols of $\widetilde{g}$ on
$\mathcal{C}$}\label{A2} The trace of the metric on $\mathcal{C}$
induces the following expresions for the Christoffel symbols of the
metric on $\mathcal{C}$.
\begin{equation*}
    \widetilde{\Gamma}^0_{00} =
  \frac{1}{2}\widetilde{g}^{01}(2\partial_0
\widetilde{g}_{01}-\partial_1\widetilde{g}_{01}),\;
  \widetilde{\Gamma}^0_{01}  =
 \frac{1}{2}\widetilde{g}^{01}\partial_0
\widetilde{g}_{11},
\end{equation*}
\begin{equation*}
    \widetilde{\Gamma}^0_{0C} =
\frac{1}{2}\widetilde{g}^{01}(\partial_0
\widetilde{g}_{1C}+\partial_C \widetilde{g}_{01}),
  \widetilde{\Gamma}^0_{11} =0  ,
  \widetilde{\Gamma}^0_{1C} =0  ,
  \widetilde{\Gamma}^0_{CD} =-\frac{1}{2}\widetilde{g}^{01}\partial_1
\widetilde{g}_{CD},
\end{equation*}
\begin{equation*}
  \widetilde{\Gamma}^1_{00} =
  \frac{1}{2}\widetilde{g}^{01}(\partial_1
\widetilde{g}_{10}-\partial_0 \widetilde{g}_{01}),
\widetilde{\Gamma}^1_{01} =
\frac{1}{2}\widetilde{g}^{10}(\partial_1\widetilde{g}_{01}-
\partial_0\widetilde{g}_{11}),\;
  \widetilde{\Gamma}^1_{0C} =-
  \frac{1}{2}\widetilde{g}^{01}\partial_0
\widetilde{g}_{1C},
\end{equation*}
\begin{equation*}
    \widetilde{\Gamma}^1_{11} =
\frac{1}{2}\widetilde{g}^{10}(2\partial_1\widetilde{g}_{01}-\partial_0
\widetilde{g}_{11}),\; \widetilde{\Gamma}^1_{1C} =
\frac{1}{2}\widetilde{g}^{10}(\partial_C\widetilde{g}_{01}-\partial_0
\widetilde{g}_{1C}),
\end{equation*}
\begin{equation*}
    \widetilde{\Gamma}^1_{CD} =
\frac{1}{2}\widetilde{g}^{10}(\partial_1 \widetilde{g}_{CD}
-\partial_0 \widetilde{g}_{CD}),\;\widetilde{\Gamma}^C_{AB} =
  \frac{1}{2}\widetilde{g}^{CD}(\partial_A
\widetilde{g}_{DB}+
\partial_B
\widetilde{g}_{DA}-\partial_D \widetilde{g}_{AB}),
\end{equation*}
\begin{equation*}
    \widetilde{\Gamma}^C_{00} =-
  \frac{1}{2}\widetilde{g}^{CB}\partial_B \widetilde{g}_{01},\;
  \widetilde{\Gamma}^C_{01} = \frac{1}{2}\widetilde{g}^{CB}(\partial_0
\widetilde{g}_{B1}-\partial_B\widetilde{g}_{01}),
\end{equation*}
\begin{equation*}
  \widetilde{\Gamma}^C_{0D} =\frac{1}{2}\widetilde{g}^{CB}\partial_0
\widetilde{g}_{BD},\; \widetilde{\Gamma}^C_{1D} =
\frac{1}{2}\widetilde{g}^{CB}\partial_1 \widetilde{g}_{BD},\;
  \widetilde{\Gamma}^C_{11} =0.
\end{equation*}
\section{ Ricci tensor on $\mathcal{C}$}\label{A3} Using the
collected formula and the expressions of the Christoffel symbols of
the metric above, one establishes the expressions of the components
of the Ricci tensor on $\mathcal{C}$, following is the sketch.
 \subsection{Computation of $\widetilde{R}_{11}$}
 By definition
 \begin{eqnarray*}
  \widetilde{R}_{11} &=& \partial_1\widetilde{\Gamma}_{\gamma
  1}^\gamma-\partial_\gamma \widetilde{\Gamma}_{1
  1}^\gamma+\widetilde{\Gamma}_{1
  \gamma}^\delta\widetilde{\Gamma}_{\delta
  1}^\gamma-\widetilde{\Gamma}_{\gamma
  \delta}^\gamma\widetilde{\Gamma}_{1
  1}^\delta.
\end{eqnarray*}Since $\widetilde{\Gamma}^C_{11}=0$ according to the expressions of appendix
\ref{A2}, the term $\partial_1\widetilde{\Gamma}_{\gamma
  1}^\gamma-\partial_\gamma \widetilde{\Gamma}_{1
  1}^\gamma$ splits as:
\begin{eqnarray*}
  \partial_1\widetilde{\Gamma}_{\gamma
  1}^\gamma-\partial_\gamma \widetilde{\Gamma}_{1
  1}^\gamma &=& \partial_1\widetilde{\Gamma}_{01}^0-\partial_0 \widetilde{\Gamma}_{1
  1}^0+\partial_1\widetilde{\Gamma}_{1
  1}^1-\partial_1 \widetilde{\Gamma}_{1
  1}^1+\partial_1\widetilde{\Gamma}_{C
  1}^C-\partial_C \widetilde{\Gamma}_{1
  1}^C\\
   &=& \partial_1\widetilde{\Gamma}_{0
  1}^0-\partial_0 \widetilde{\Gamma}_{1
  1}^0+\partial_1\widetilde{\Gamma}_{C
  1}^C,
\end{eqnarray*} On the other hand, from the expression of
$\widetilde{\Gamma}^0_{01}$ it results that:
\begin{equation}\label{x43}
    \partial_1\widetilde{\Gamma}^0_{01}=\frac{1}{2}\partial_1
    [\widetilde{g}^{01}\partial_0 \widetilde{g}_{11}]=\frac{1}{2}(\partial_1
    \widetilde{g}^{01})(\partial_0 \widetilde{g}_{11})+\frac{1}{2}
    \widetilde{g}^{01}\partial_1 [\partial_0 \widetilde{g}_{11}].
\end{equation}The term $\partial_0 \widetilde{\Gamma}_{11}^0$
requires more attention due to the normal character of the
derivative $\partial_0$. One has:
\begin{eqnarray*}
  \partial_0 \widetilde{\Gamma}_{11}^0 &=&
  \partial_0[\frac{1}{2}\widetilde{g}^{0\gamma}
  (2\partial_1\widetilde{g}_{\gamma 1}-\partial_\gamma \widetilde{g}_{11})] \\
   &=&\frac{1}{2}
  \partial_0(\widetilde{g}^{0\gamma})
  (2\partial_1\widetilde{g}_{\gamma 1}-\partial_\gamma \widetilde{g}_{11})+
 \frac{1}{2}
  \widetilde{g}^{0\gamma}
  (2\partial_{01}^2\widetilde{g}_{\gamma 1}-\partial^2_{0\gamma} \widetilde{g}_{11}) \\
   &=&\frac{1}{2}
  \partial_0(\widetilde{g}^{00})
  (2\partial_1\widetilde{g}_{0 1}-\partial_0 \widetilde{g}_{11})+\frac{1}{2}
  \partial_0(\widetilde{g}^{01})
  (2\partial_1\widetilde{g}_{1 1}-\partial_1 \widetilde{g}_{11})+\\
  &&\frac{1}{2}
  \partial_0(\widetilde{g}^{0C})
  (2\partial_1\widetilde{g}_{C 1}-\partial_C \widetilde{g}_{11})+\frac{1}{2}
  \widetilde{g}^{01}
  (2\partial_{01}^2\widetilde{g}_{11}-\partial^2_{01}
\widetilde{g}_{11}),\hbox{}
\end{eqnarray*}and after more simplifications thanks to
properties of the metric and its inverse (\ref{c4}), (\ref{y4}), one
ends up with:
\begin{equation}\label{x41}
  \partial_0 \widetilde{\Gamma}_{11}^0 =\frac{1}{2}
  \partial_0(\widetilde{g}^{00})
  (2\partial_1\widetilde{g}_{0 1}-\partial_0 \widetilde{g}_{11})+\frac{1}{2}
  \widetilde{g}^{01}
  \partial_1(\partial_0\widetilde{g}_{11}).
\end{equation}The expression of $\widetilde{\Gamma}^C_{C1}$ (appendix
\ref{A2}) implies:
\begin{equation}\label{x42}
\partial_1\widetilde{\Gamma}^C_{C1}=\frac{1}{2}\partial_1(\widetilde{g}^{CB}\partial_1
\widetilde{g}_{CB}).
\end{equation}Now, one deals with the term
$\widetilde{E}_{11} =\widetilde{\Gamma}^\delta_{1\gamma}
  \widetilde{\Gamma}^\gamma_{\delta 1}-\widetilde{\Gamma}^\gamma_{\gamma\delta}
  \widetilde{\Gamma}^\delta_{1 1}$, since $\widetilde{\Gamma}^0_{11}=0=\widetilde{\Gamma}^C_{11}=
  \widetilde{\Gamma}^0_{1C}$
(appendix \ref{A2}), this terms simplifies in:
\begin{equation*}
    \widetilde{E}_{11}=
    \widetilde{\Gamma}^0_{10}
  \widetilde{\Gamma}^0_{0 1}+\widetilde{\Gamma}^C_{1D}
  \widetilde{\Gamma}^D_{C 1}-\widetilde{\Gamma}^0_{0 1}
  \widetilde{\Gamma}^1_{1
  1}-\widetilde{\Gamma}^C_{C 1}
  \widetilde{\Gamma}^1_{1
  1},
\end{equation*}and yields:
\begin{equation}\label{x40}
    \widetilde{E}_{11}=\frac{1}{4}(\widetilde{g}^{01}\partial_0
    \widetilde{g}_{11})^2+\frac{1}{4}(\widetilde{g}^{CB}\partial_1\widetilde{g}_{BD})
    (\widetilde{g}^{DE}\partial_1\widetilde{g}_{CE})-
    \frac{1}{2}(\widetilde{g}^{01}\partial_0\widetilde{g}_{11}+
    \widetilde{g}^{CB}\partial_1\widetilde{g}_{CB})\widetilde{\Gamma}^1_{11}.
\end{equation}
Combining the relations (\ref{x43})-(\ref{x40}) it results
\begin{equation*}
    \widetilde{R}_{11}= \frac{1}{2}(\partial_1
    \widetilde{g}^{01})\partial_0 \widetilde{g}_{11}-
    \frac{1}{2}(\partial_0 \widetilde{g}^{00})(2\partial_1
\widetilde{g}_{01}-\partial_0
\widetilde{g}_{11})+\frac{1}{2}\partial_1(\widetilde{g}^{CB}\partial_1
\widetilde{g}_{CB})
\end{equation*}
\begin{equation}
      +\frac{1}{4}(\widetilde{g}^{01}\partial_0
    \widetilde{g}_{11})^2+\frac{1}{4}(\widetilde{g}^{CB}\partial_1\widetilde{g}_{BD})
    (\widetilde{g}^{DE}\partial_1\widetilde{g}_{CE})-
    \frac{1}{2}(\widetilde{g}^{01}\partial_0\widetilde{g}_{11}+
    \widetilde{g}^{CB}\partial_1\widetilde{g}_{CB})\widetilde{\Gamma}^1_{11}.
\end{equation}Using the expression of
$\partial_0\widetilde{g}^{00}$ (appendix \ref{A1}), the one of
$\widetilde{\Gamma}^1_{11} $ (appendix \ref{A2} and that
$\partial_1(\widetilde{g}^{01}\widetilde{g}_{01})=0$ the expression
of $\widetilde{R}_{11}$ resumes in:
\begin{equation*}
   \widetilde{ R}_{11}=\frac{1}{4}\widetilde{g}^{01}\widetilde{g}^{CB}\partial_1
\widetilde{g}_{CB}\partial_0\widetilde{g}_{11}+
\frac{1}{2}\partial_1(\widetilde{g}^{CB}\partial_1
\widetilde{g}_{CB})+
\end{equation*}
\begin{equation}\label{y34}
 \frac{1}{4}(\widetilde{g}^{CB}\partial_1\widetilde{g}_{BD})
    (\widetilde{g}^{DE}\partial_1\widetilde{g}_{CE})-
       \frac{1}{2}\widetilde{g}^{01}\partial_1\widetilde{g}_{01} \widetilde{g}^{CB}\partial_1
\widetilde{g}_{CB}.
\end{equation}
\subsection{Computation of $\widetilde{R}_{1A}$} By definition
\begin{eqnarray*}
  \widetilde{R}_{1A} &=& \partial_A\widetilde{\Gamma}_{\gamma
  1}^\gamma-\partial_\gamma \widetilde{\Gamma}_{1
  A}^\gamma+\widetilde{\Gamma}_{A
  \gamma}^\delta\widetilde{\Gamma}_{\delta
  1}^\gamma-\widetilde{\Gamma}_{\gamma
  \delta}^\gamma\widetilde{\Gamma}_{1
  A}^\delta.
\end{eqnarray*}Furthermore:
\begin{eqnarray*}
  \partial_\gamma \widetilde{\Gamma}^{\gamma}_{1A} &=&
   \partial_0\widetilde{\Gamma }^0_{1A}+\partial_1\widetilde{\Gamma }^1_{1A}+
   \partial_C\widetilde{\Gamma }^C_{1A}.
\end{eqnarray*}Using the properties of the metric and its inverse
(\ref{c4}), (\ref{y4}), the term $ \partial_0\widetilde{\Gamma
}^0_{1A}$ simplifies in:
\begin{equation}\label{x44}
     \partial_0\widetilde{\Gamma }^0_{1A}  =\frac{1}{2}(\partial_0\widetilde{g}^{00})
  (\partial_A \widetilde{g}_{01}-\partial_0\widetilde{g}_{1A})+\frac{1}{2}
  (\partial_0\widetilde{g}^{0C})\partial_1 \widetilde{g}_{CA}+\frac{1}{2}
  \widetilde{g}^{01}
   \partial^2_{0A}\widetilde{g}_{1 1}.
\end{equation}The other terms resume in:
\begin{equation}\label{x45}
 \partial_A\widetilde{\Gamma}_{\gamma
  1}^\gamma
   =\frac{1}{2}\partial_A[\widetilde{g}^{01}
   \partial_0\widetilde{g}_{11}+2\widetilde{\Gamma}^1_{11}+
   \widetilde{g}^{CB}\partial_1 \widetilde{g}_{CB}],
\end{equation}
\begin{equation}\label{x46}
     \partial_1\widetilde{\Gamma}^1_{1A}  = \frac{1}{2}(\partial_1\widetilde{g}^{01})
(\partial_A\widetilde{g}_{01}-
  \partial_0\widetilde{g}_{1A})+\frac{1}{2}\widetilde{g}^{01}
  (\partial^2_{1A}\widetilde{g}_{01}-
  \partial^2_{10}\widetilde{g}_{1A}),
\end{equation}
\begin{equation}\label{x47}
  \partial_C\widetilde{\Gamma}^C_{1A} = \frac{1}{2}\partial_C
  [\widetilde{g}^{CB}\partial_1 \widetilde{g}_{AB}].
\end{equation}About the term $\widetilde{E}_{1A}\equiv
 \widetilde{\Gamma}^\delta_{A\gamma}
  \widetilde{\Gamma}^\gamma_{\delta 1}-\widetilde{\Gamma}^\gamma_{\gamma\delta}
  \widetilde{\Gamma}^\delta_{1 A}$, it simplifies in:
  \begin{equation*}
    \widetilde{E}_{1A}=
        \frac{1}{4}(\widetilde{g}^{01})^2(\partial_0
  \widetilde{g}_{11})(\partial_0 \widetilde{g}_{1A}+\partial_A \widetilde{g}_{01})+
    \frac{1}{2}\widetilde{g}^{01}(\partial_A\widetilde{g}_{01}-
  \partial_0 \widetilde{g}_{1A})\widetilde{\Gamma}^1_{11}+
  \end{equation*}
\begin{equation*}
\frac{1}{4}\widetilde{g}^{01}(\partial_C
\widetilde{g}_{01}-\partial_0 \widetilde{g}_{1C})
  (\widetilde{g}^{CB}\partial_1\widetilde{g}_{BA})
        -\frac{1}{4}\widetilde{g}^{01}(\partial_1\widetilde{g}_{CA})\widetilde{g}^{CB}
   (\partial_0 \widetilde{g}_{B1}-\partial_B \widetilde{g}_{01}) +
  \end{equation*}
\begin{equation*}
   \frac{1}{4}\widetilde{g}^{DE}\widetilde{g}^{CB}(\partial_1
   \widetilde{g}_{DB})
   (\partial_A \widetilde{g}_{EC}+\partial_C \widetilde{g}_{EA}-\partial_E\widetilde{g}_{AC})
    -\frac{1}{4}\widetilde{g}^{01}\widetilde{g}^{\gamma\delta}
   (\partial_1 \widetilde{g}_{\gamma\delta})(\partial_A \widetilde{g}_{01}-
   \partial_0 \widetilde{g}_{1A})
\end{equation*}
\begin{equation}\label{x47}
       -\frac{1}{4}\widetilde{g}^{\gamma\delta}
   (\partial_C \widetilde{g}_{\gamma\delta})(\widetilde{g}^{CB}
   \partial_1 \widetilde{g}_{BA}).
      \end{equation}Combining the expressions (\ref{x44})-(\ref{x47}) one
obtains:
\begin{equation*}
\widetilde{R}_{1A}=
\frac{1}{2}\widetilde{g}^{01}\partial_1(\partial_0
\widetilde{g}_{1A})+\frac{1}{2}(\frac{1}{2}\widetilde{g}^{CB}
\partial_1\widetilde{g}_{CB}-\widetilde{g}^{01}\partial_1\widetilde{g}_{01})\partial_0
\widetilde{g}_{1A}+
\end{equation*}
\begin{equation*}
    \frac{1}{2}\widetilde{g}^{01}(\widetilde{g}^{01}\partial_0\widetilde{g}_{11}-
   \widetilde{g}^{01}\partial_1\widetilde{g}_{01}- \frac{1}{2}\widetilde{g}^{CB}
\partial_1\widetilde{g}_{CB})\partial_A\widetilde{g}_{01}
- \frac{1}{2}\partial_C(\widetilde{g}^{CB}
\partial_1\widetilde{g}_{AB})+
\end{equation*}
\begin{equation*}
 \frac{1}{2}\partial_A(\widetilde{g}^{CB}
\partial_1\widetilde{g}_{CB})   - \frac{1}{2}\partial_A(\partial_0\widetilde{g}_{11})
    -          \frac{1}{4}(\widetilde{g}^{EF}\partial_C
          \widetilde{g}_{EF})(\widetilde{g}^{CB}\partial_1
          \widetilde{g}_{BA})+
\end{equation*}
\begin{equation}\label{y23}
      \frac{1}{4}\widetilde{g}^{DE}\widetilde{g}^{CB}(\partial_1 \widetilde{g}_{DB})
       (\partial_A\widetilde{g}_{EC}+\partial_C\widetilde{g}_{EA}-
  \partial_E\widetilde{g}_{AC}).
\end{equation}
\subsection{Computation of $\widetilde{R}_{AB}$} By definition:
\begin{eqnarray*}
  \widetilde{R}_{AB} &=& \partial_B\widetilde{\Gamma}_{\gamma
  A}^\gamma-\partial_\gamma
  \widetilde{\Gamma}_{AB}^\gamma+\widetilde{\Gamma}_{A
  \gamma}^\delta\widetilde{\Gamma}_{\delta
  B}^\gamma-\widetilde{\Gamma}_{\gamma
  \delta}^\gamma\widetilde{\Gamma}_{A
  B}^\delta.
\end{eqnarray*}The term $\partial_\gamma
\widetilde{\Gamma}^{\gamma}_{AB}$ splits as:
\begin{eqnarray*}
  \partial_\gamma \widetilde{\Gamma}^{\gamma}_{AB} &=&
   \partial_0\widetilde{\Gamma }^0_{AB}+\partial_1\widetilde{\Gamma }^1_{AB}+
   \partial_C\widetilde{\Gamma }^C_{AB}.
\end{eqnarray*}The properties of the trace of the metric and its inverse (\ref{c4}), (\ref{y4}),
induce that:
\begin{equation*}
\partial_0\widetilde{\Gamma }^0_{BA}=
    \frac{1}{2}(\partial_0\widetilde{g}^{00})
   (-\partial_0\widetilde{g}_{BA})+
  \frac{1}{2}(\partial_0\widetilde{g}^{0C})
   (\partial_B\widetilde{g}_{CA}+\partial_A\widetilde{g}_{CB }-
  \partial_C\widetilde{g}_{BA})-
\end{equation*}
\begin{equation}\label{x48}
 \frac{1}{2}(\partial_0\widetilde{g}^{01})
   (\partial_1\widetilde{g}_{BA})+    \frac{1}{2}\widetilde{g}^{01}
  (\partial^2_{0B}\widetilde{g}_{1 A}+
  \partial^2_{0A}\widetilde{g}_{1B }-
  \partial^2_{01}\widetilde{g}_{BA}),
    \end{equation}
\begin{equation}\label{x49}
  \partial_B\widetilde{\Gamma}^\gamma_{\gamma A} =
  \frac{1}{2}\partial_B
  [\widetilde{g}^{\gamma\delta}(\partial_\gamma
   \widetilde{g}_{\delta A}+\partial_A \widetilde{g}_{\gamma \delta}-
   \partial_\delta \widetilde{g}_{\gamma
   A})]=\frac{1}{2}\underset{(\gamma,\delta)\in
   \{(0,1);(1,0);(A,B)\}}{\underbrace{\partial_B[\widetilde{g}^{\gamma\delta}\partial_A
   \widetilde{g}_{\gamma\delta }}]},
\end{equation}
\begin{equation*}
\partial_1\widetilde{\Gamma}^1_{AB}  = -\frac{1}{2}(\partial_1\widetilde{g}^{01})(
  \partial_0\widetilde{g}_{AB})-\frac{1}{2}\widetilde{g}^{01}
  ( \partial^2_{10}\widetilde{g}_{AB})+
\end{equation*}
\begin{equation}\label{x50}
   \frac{1}{2}(\partial_1\widetilde{g}^{01})(
  \partial_1\widetilde{g}_{AB})+\frac{1}{2}\widetilde{g}^{01}
  ( \partial^2_{11}\widetilde{g}_{AB}),
\end{equation}
\begin{equation}\label{x51}
  \partial_C\widetilde{\Gamma}^C_{BA} =
  \frac{1}{2}\partial_C[\widetilde{g}^{CD}(\partial_B\widetilde{g}_{DA}+
  \partial_A\widetilde{g}_{DB}-\partial_D \widetilde{g}_{AB})];
\end{equation} Now, one is interested in
$\widetilde{E}_{AB} = \widetilde{\Gamma}^\delta_{A\gamma}
  \widetilde{\Gamma}^\gamma_{\delta B}-\widetilde{\Gamma}^\gamma_{\gamma\delta}
  \widetilde{\Gamma}^\delta_{AB}$. Straightforward computations
  using the properties of the trace of the metric and its inverse
  (\ref{c4}), (\ref{y4}), and some of the expressions of the Christoffel
  symbols of the metric on $\mathcal{C}$ (appendix \ref{A2}), lead
  to
  the following non-explicit expression:
\begin{equation*}
\widetilde{E}_{AB} =
        \frac{1}{4}(\widetilde{g}^{01})^2
    (\partial_0 \widetilde{g}_{1A}+\partial_A \widetilde{g}_{01})
    (\partial_0 \widetilde{g}_{1B}-\partial_B \widetilde{g}_{01})
    -\frac{1}{4}(\widetilde{g}^{CD}\partial_0\widetilde{g}_{DA})
    (\widetilde{g}^{01}\partial_1 \widetilde{g}_{CB})+
\end{equation*}
\begin{equation*}
\frac{1}{4}(\widetilde{g}^{01})^2
    (\partial_A \widetilde{g}_{01}-\partial_0\widetilde{g}_{1A})
    (\partial_B \widetilde{g}_{01}-\partial_0\widetilde{g}_{1B})
    -\frac{1}{4}\widetilde{g}^{01}\widetilde{g}^{CD}
   ( \partial_1\widetilde{g}_{DA})(\partial_0\widetilde{g}_{CB}-
   \partial_1\widetilde{g}_{CB})
\end{equation*}
\begin{equation*}
    -\frac{1}{4}\widetilde{g}^{01}\widetilde{g}^{CD}
   ( \partial_1\widetilde{g}_{AC})(\partial_0\widetilde{g}_{BD})
   -\frac{1}{4}\widetilde{g}^{01}\widetilde{g}^{CD}
   ( \partial_0\widetilde{g}_{AC}-\partial_1\widetilde{g}_{AC})
   (\partial_1\widetilde{g}_{BD})
\end{equation*}
\begin{equation*}
    +\widetilde{\Gamma}^D_{AC}\widetilde{\Gamma}^C_{DB}+
     \frac{1}{4}(\widetilde{g}^{01})\partial_1
     \widetilde{g}_{AB}\left[\widetilde{g}^{01}(2\partial_0
     \widetilde{g}_{01}-\partial_0\widetilde{g}_{11})+\widetilde{g}^{CE}
     \partial_0\widetilde{g}_ {CE}\right]+
\end{equation*}
\begin{equation*}
+\frac{1}{4}(\widetilde{g}^{01})\left[ \widetilde{g}^{01}\partial_0
     \widetilde{g}_{11}+2\widetilde{\Gamma}^1_{11}+
     \widetilde{g}^{CB}(\partial_1
     \widetilde{g}_{BC})\right](\partial_0 \widetilde{g}_{AB}-
     \partial_1\widetilde{g}_{AB})
\end{equation*}
\begin{equation}\label{x52}
         -[     \widetilde{g}^{01}(\partial_C\widetilde{g}_{01})+
     \widetilde{\Gamma}^D_{DC}]\widetilde{\Gamma}^C_{AB}.
\end{equation}Combining the relations (\ref{x48})-(\ref{x52}), it follows:
\begin{equation*}
     \widetilde{R}_{AB} =\widetilde{g}^{01}
  \partial^2_{01}\widetilde{g}_{BA}+\frac{1}{2}\widetilde{g}^{01}(-\widetilde{g}^{01}\partial_0
  \widetilde{g}_{11}+\frac{1}{2}\widetilde{g}^{CB}\partial_1\widetilde{g}_{CB})\partial_0\widetilde{g}_{AB}+
\end{equation*}
\begin{equation*}
    \frac{1}{2}\widetilde{g}^{01}\left[-\widetilde{g}^{CD}(\delta_A^E\partial_1\widetilde{g}_{DB}
    +\delta_B^E\partial_1\widetilde{g}_{DA})+\frac{1}{2}\widetilde{g}^{CE}\partial_1\widetilde{g}_{AB}\right]
    \partial_0\widetilde{g}_{CE}+
\end{equation*}
\begin{equation*}
   \frac{1}{4}(\widetilde{g}^{01})^2\partial_0\widetilde{g}_{11}\partial_1\widetilde{g}_{AB}+
    \frac{1}{2}\widetilde{g}^{01}\widetilde{g}^{CD}(\partial_0\widetilde{g}_{1D})
   (\partial_B\widetilde{g}_{CA}+\partial_A\widetilde{g}_{CB }
  -
  \partial_C\widetilde{g}_{BA}) -
\end{equation*}
\begin{equation*}
    \frac{1}{2}\widetilde{g}^{01}
  (\partial^2_{0B}\widetilde{g}_{1 A}+
  \partial^2_{0A}\widetilde{g}_{1B })-  \frac{1}{2}\partial_C[\widetilde{g}^{CD}(\partial_B\widetilde{g}_{DA}+
  \partial_A\widetilde{g}_{DB}-\partial_D \widetilde{g}_{AB})] +
\end{equation*}
\begin{equation*}
      \frac{1}{2}\underset{(\gamma,\delta)\in
   \{(0,1);(1,0);(A,B)\}}{\underbrace{\partial_B[\widetilde{g}^{\gamma\delta}\partial_A
   \widetilde{g}_{\gamma\delta }}]}-
   \frac{1}{2}(\widetilde{g}^{01})^2
    (\partial_0\widetilde{g}_{1A})
    (\partial_B \widetilde{g}_{01}-\partial_0\widetilde{g}_{1B})+
\end{equation*}
\begin{equation*}
\frac{1}{2}\widetilde{g}^{01}\widetilde{g}^{CD}
   ( \partial_1\widetilde{g}_{AC})(\partial_1\widetilde{g}_{BD})-
   \frac{1}{4}\widetilde{g}^{01}\widetilde{g}^{CE}
   ( \partial_1\widetilde{g}_{CE})(\partial_1\widetilde{g}_{AB})+
\end{equation*}
\begin{equation}\label{y32}
\widetilde{\Gamma}^D_{AC}\widetilde{\Gamma}^C_{DB}-[
\widetilde{g}^{01}(\partial_C\widetilde{g}_{01})+
     \widetilde{\Gamma}^D_{DC}]\widetilde{\Gamma}^C_{AB}.
\end{equation}
\subsection{Computation of $\widetilde{R}_{01}$} By definition:
\begin{eqnarray*}
  \widetilde{R}_{01} &=& \partial_1\widetilde{\Gamma}_{\gamma
  0}^\gamma-\partial_\gamma
  \widetilde{\Gamma}_{01}^\gamma+\widetilde{\Gamma}_{1
  \gamma}^\delta\widetilde{\Gamma}_{\delta
  0}^\gamma-\widetilde{\Gamma}_{\gamma
  \delta}^\gamma\widetilde{\Gamma}_{0
  1}^\delta.
\end{eqnarray*}Thanks to  the same properties as previously
mentioned, one has:
\begin{equation}\label{x53}
  \partial_1\widetilde{\Gamma}_{\gamma
  0}^\gamma  =\partial_1[\widetilde{g}^{01}
  \partial_0 \widetilde{g}_{01}]+\frac{1}{2}\partial_1[\widetilde{g}^{11}
  \partial_0 \widetilde{g}_{11}]+\frac{1}{2}\partial_1[\widetilde{g}^{AB}
  \partial_0 \widetilde{g}_{AB}],
\end{equation}
\begin{equation*}
\partial_\gamma \widetilde{\Gamma}^{\gamma}_{10}=
   \frac{1}{2}(\partial_0\widetilde{g}^{00})
  \partial_1\widetilde{g}_{0 0}+\frac{1}{2}(\partial_1\widetilde{g}^{10})
  \partial_1\widetilde{g}_{0 0}+\frac{1}{2}(\partial_1\widetilde{g}^{11})
  \partial_0\widetilde{g}_{11}   +
\end{equation*}
\begin{equation*}
\frac{1}{2}(\partial_0\widetilde{g}^{01})\partial_0\widetilde{g}_{11}+
\frac{1}{2}(\partial_0\widetilde{g}^{0C})(\partial_0\widetilde{g}_{1C}-\partial_C
\widetilde{g}_{10})+
\end{equation*}
\begin{equation*}
\frac{1}{2}(\partial_D\widetilde{g}^{DC})
  (  \partial_0\widetilde{g}_{1C}-\partial_C
  \widetilde{g}_{10})+
          \frac{1}{2}\widetilde{g}^{01}\partial^2_{0
  0}\widetilde{g}_{11}+\frac{1}{2}\widetilde{g}^{01}\partial^2_{11}
  \widetilde{g}_{00}+
\end{equation*}
\begin{equation}\label{x54}
      \frac{1}{2}\widetilde{g}^{11}\partial^2_{1
  0}\widetilde{g}_{11}+\frac{1}{2}\widetilde{g}^{CD}
  (  \partial^2_{C 0}\widetilde{g}_{1D}-\partial^2_{CD}
  \widetilde{g}_{10}).
\end{equation}Now, setting
 $\widetilde{E}_{01} \equiv \widetilde{\Gamma}^\delta_{1\gamma}
  \widetilde{\Gamma}^\gamma_{\delta 0}-\widetilde{\Gamma}^\gamma_{\gamma\delta}
  \widetilde{\Gamma}^\delta_{01}$, one obtains after straightforward computations:
\begin{equation*}
 \widetilde{E}_{01}=
         \frac{1}{4}\widetilde{g}^{01}\widetilde{g}^{CB}
    (\partial_0 \widetilde{g}_{B1}-2\partial_B \widetilde{g}_{01})
    (\partial_0 \widetilde{g}_{1C}+\partial_C \widetilde{g}_{01})
    +\frac{1}{2}\widetilde{\Gamma}_{11}^1
    (\partial_1\widetilde{g}_{01}-\partial_0
    \widetilde{g}_{11})-
\end{equation*}
\begin{equation*}
  \frac{1}{4}\widetilde{g}^{01}\widetilde{g}^{BC}\partial_B\widetilde{g}_{00}
  (\partial_0\widetilde{g}_{1C}+\partial_C\widetilde{g}_{01})  -\frac{1}{4}\widetilde{g}^{01}\widetilde{g}^{CB}
    \partial_0 \widetilde{g}_{C1}(\partial_0 \widetilde{g}_{B1}
    -\partial_B \widetilde{g}_{01})+
\end{equation*}
\begin{equation*}
    \frac{1}{4}\widetilde{g}^{DB}\widetilde{g}^{CE}
    (\partial_0 \widetilde{g}_{BC})
    (\partial_0 \widetilde{g}_{DE})
    -\frac{1}{4}\widetilde{g}^{01}\widetilde{g}^{CB}
   ( \partial_0\widetilde{g}_{CB})(\partial_0\widetilde{g}_{11}) -
\end{equation*}
\begin{equation*}
    \frac{1}{4}(\partial_1\widetilde{g}_{01}-\partial_0
    \widetilde{g}_{11})(\widetilde{g}^{01}\partial_0\widetilde{g}_{11}
    +\widetilde{g}^{CB}
    \partial_1\widetilde{g}_{CB}-2\widetilde{\Gamma}_{11}^1) -
\end{equation*}
\begin{equation}\label{x55}
   \frac{1}{2}\widetilde{g}^{CB}\{ \widetilde{g}^{01}
   ( \partial_C\widetilde{g}_{01})
        +\frac{1}{2}\widetilde{g}^{DE}
   ( \partial_C\widetilde{g}_{DE}+\partial_D\widetilde{g}_{CE}-
   \partial_E\widetilde{g}_{DC})\}(\partial_0\widetilde{g}_{B1}-
   \partial_B\widetilde{g}_{01}).
\end{equation}
Exploiting the relations (\ref{x53})-(\ref{x55}), it follows:
\begin{equation*}
    \widetilde{R}_{01}=
        -\frac{1}{2}\widetilde{g}^{01}\partial^2_{0
  0}\widetilde{g}_{11}+\widetilde{g}^{01}\partial_1(\partial_0\widetilde{g}_{01})-
  (\widetilde{g}^{01})^2(\partial_1\widetilde{g}_{01})(\partial_0\widetilde{g}_{01})-
\end{equation*}
\begin{equation*}
  \frac{7}{4}(\widetilde{g}^{01})^2(\partial_1\widetilde{g}_{01})(\partial_0\widetilde{g}_{11})-
  \frac{1}{2}\widetilde{g}^{01}\partial_1(\partial_0\widetilde{g}_{11})+
    \frac{1}{2}\partial_1(\widetilde{g}^{AB}\partial_0\widetilde{g}_{AB})
    +
\end{equation*}
\begin{equation*}
    (\widetilde{g}^{01}\partial_1\widetilde{g}_{01})^2-\frac{1}{2}
    (\widetilde{g}^{01})^2(\partial_0\widetilde{g}_{11})
    (\partial_0\widetilde{g}_{01})+\frac{1}{4}(\widetilde{g}^{01}\partial_0\widetilde{g}_{11})^2-
\end{equation*}
\begin{equation*}
\frac{1}{4}\widetilde{g}^{01}\widetilde{g}^{CB}
(\partial_1\widetilde{g}_{CB})(\partial_1\widetilde{g}_{01}-
\partial_0\widetilde{g}_{11})+\frac{1}{2}\widetilde{g}^{01}\widetilde{g}^{CB}\partial_0\widetilde{g}_{1C}
\partial_0\widetilde{g}_{1B}-
\end{equation*}
\begin{equation*}
    -
 \frac{1}{2}(\partial_D\widetilde{g}^{DC})
  (  \partial_0\widetilde{g}_{1C}-\partial_C
  \widetilde{g}_{10})-\frac{1}{2}\widetilde{g}^{01}\partial^2_{11}
  \widetilde{g}_{01}+\frac{1}{2}\widetilde{g}^{01}\partial^2_{1
  0}\widetilde{g}_{11}-
\end{equation*}
\begin{equation*}
  \frac{1}{2}\widetilde{g}^{CD}
  (  \partial^2_{C 0}\widetilde{g}_{1D}-\partial^2_{CD}
  \widetilde{g}_{10})-
  \widetilde{g}^{01}\widetilde{g}^{CB}\partial_C\widetilde{g}_{01}\partial_0\widetilde{g}_{1B}+
  \end{equation*}
  \begin{equation*}
    \frac{1}{4}\widetilde{g}^{DB}\widetilde{g}^{CE}
    (\partial_0 \widetilde{g}_{BC})
    (\partial_0 \widetilde{g}_{DE})
    -
     \frac{1}{4}\widetilde{g}^{01}\widetilde{g}^{CB}
   ( \partial_0\widetilde{g}_{CB})(\partial_0\widetilde{g}_{11})  -
  \end{equation*}
\begin{equation}\label{y33}
 \frac{1}{2}\widetilde{g}^{CB}(\partial_0\widetilde{g}_{B1}-
   \partial_B\widetilde{g}_{01})\widetilde{\Gamma}^D_{DC}.
\end{equation}
\subsection{Computation of $\widetilde{R}_{00}$} By definition:
\begin{eqnarray*}
  \widetilde{R}_{00} &=& \partial_0\widetilde{\Gamma}_{\gamma
  0}^\gamma-\partial_\gamma \widetilde{\Gamma}_{0
  0}^\gamma+\widetilde{\Gamma}_{0
  \gamma}^\delta\widetilde{\Gamma}_{\delta
  0}^\gamma-\widetilde{\Gamma}_{\gamma
  \delta}^\gamma\widetilde{\Gamma}_{00}^\delta.
\end{eqnarray*}After simplifications, one obtains:
\begin{equation}\label{x56}
\partial_0 \widetilde{\Gamma}^\gamma_{\gamma 0}  =
\frac{1}{2}(\partial_0 \widetilde{g}^{\gamma\delta})\partial_0
        \widetilde{g}_{\gamma\delta}+
\widetilde{g}^{01}\partial_{00}^2\widetilde{g}_{10}+ \frac{1}{2}
   \widetilde{g}^{11}
   \partial^2_{00} \widetilde{g}_{11}
 +\frac{1}{2}\widetilde{g}^{AB}\partial_{00}^2\widetilde{g}_{AB}.
\end{equation}On the other hand one has:
\begin{equation*}
    \partial_\gamma \widetilde{\Gamma}^\gamma_{00}=
    \frac{1}{2}(\partial_\gamma\widetilde{g}^{\gamma\delta})
   (2\partial_0 \widetilde{g}_{0\delta}-\partial_\delta
   \widetilde{g}_{00})+\widetilde{g}^{01}
   \partial^2_{00} \widetilde{g}_{01}+
\end{equation*}
\begin{equation}\label{x57}
\widetilde{g}^{11}
   \partial^2_{01} \widetilde{g}_{01}-\frac{1}{2}\widetilde{g}^{11}
   \partial^2_{11} \widetilde{g}_{00}- \frac{1}{2}\widetilde{g}^{AB}
  \partial_{AB}^2\widetilde{g}_{00}.
  \end{equation}Concerning $\widetilde{E}_{00} \equiv
\widetilde{\Gamma}^\delta_{0\gamma}
  \widetilde{\Gamma}^\gamma_{\delta 0}-\widetilde{\Gamma}^\gamma_{\gamma\delta}
  \widetilde{\Gamma}^\delta_{00}$, it resumes in:
\begin{equation*}
\widetilde{E}_{00} =
     \frac{1}{4}(\widetilde{g}^{01})^2\partial_0\widetilde{g}_{11}
    (\partial_1 \widetilde{g}_{01}-\partial_0\widetilde{g}_{01})+
\frac{1}{4} (\widetilde{g}^{01})^2
    (\partial_1 \widetilde{g}_{01}-\partial_0\widetilde{g}_{11})^2-
\end{equation*}
\begin{equation*}
    \frac{1}{4}\widetilde{g}^{01}\widetilde{g}^{BC}\partial_B\widetilde{g}_{00}
  (\partial_0\widetilde{g}_{1C}+\partial_C\widetilde{g}_{01})
  -\frac{1}{2}\widetilde{g}^{01}\widetilde{g}^{CB}
(\partial_0\widetilde{g}_{1B}-\partial_B\widetilde{g}_{01})
\partial_0\widetilde{g}_{1C}
       \end{equation*}
\begin{equation*}
    -
    \frac{1}{4}
     \widetilde{g}^{01}(2\partial_0
     \widetilde{g}_{10}-\partial_1\widetilde{g}_{00})
\left[\widetilde{g}^{01}(\partial_1
     \widetilde{g}_{10}-\partial_0\widetilde{g}_{11})+
\widetilde{g}^{CB}\partial_0\widetilde{g}_{BC}\right]+
\end{equation*}
\begin{equation*}
\frac{1}{4}\widetilde{g}^{DB}\widetilde{g}^{CE}\partial_0\widetilde{g}_{BC}
    \partial_0 \widetilde{g}_{ED} -\frac{1}{4}\widetilde{g}^{01}
     (\partial_1\widetilde{g}_{01}-\partial_0 \widetilde{g}_{01})
      (2\widetilde{\Gamma}_{11}^1+\widetilde{g}^{CB}\partial_1\widetilde{g}_{BC}
     )+
\end{equation*}
\begin{equation}\label{x58}
       \frac{1}{4}\widetilde{g}^{01}\widetilde{g}^{CB}\partial_B
     \widetilde{g}_{00}(\partial_C\widetilde{g}_{10}-\partial_0\widetilde{g}_{1C})
-\frac{1}{2}\widetilde{g}^{DB}
\partial_B
\widetilde{g}_{00}\widetilde{\Gamma}^C_{CD}.
\end{equation}Combining the expressions above (\ref{x56})-(\ref{x58}) leads to:
  \begin{equation*}
     \widetilde{R}_{00}= -\frac{1}{2}
   \widetilde{g}^{01}
   \partial^2_{00} \widetilde{g}_{11}
 +\frac{1}{2}\widetilde{g}^{AB}\partial_{00}^2\widetilde{g}_{AB}+
 \widetilde{g}^{01}
   \partial^2_{01} \widetilde{g}_{01}-
   \frac{1}{2}\widetilde{g}^{AB}(
  \partial^2_{0A}
\widetilde{g}_{0B}-\partial_{AB}^2\widetilde{g}_{01})-
\end{equation*}
\begin{equation*}
   \frac{1}{2}\widetilde{g}^{01}
   \partial^2_{11} \widetilde{g}_{01} +\frac{1}{2}\widetilde{g}^{01}\left(\widetilde{g}^{01}\partial_0\widetilde{g}_{11}
    -2\widetilde{g}^{01}\partial_1\widetilde{g}_{01}+
    \widetilde{g}^{CB}(\frac{1}{2}\partial_1\widetilde{g}_{CB}-\partial_0\widetilde{g}_{CB})\right)
    \partial_0\widetilde{g}_{01}
\end{equation*}
\begin{equation*}
   - \frac{1}{4}(\widetilde{g}^{01}\partial_0\widetilde{g}_{11})^2+
    \frac{1}{2}\widetilde{g}^{01}\widetilde{g}^{AC}\partial_0\widetilde{g}_{1A}\partial_0\widetilde{g}_{1C}
-\frac{1}{4}\widetilde{g}^{DA}\widetilde{g}^{CB}\partial_0\widetilde{g}_{BA}\partial_0\widetilde{g}_{CD}+
\end{equation*}
\begin{equation*}
\frac{1}{4}(\widetilde{g}^{01})^2(\partial_1\widetilde{g}_{10})\partial_0\widetilde{g}_{11}+
\frac{1}{2}(\widetilde{g}^{01}\partial_1\widetilde{g}_{01})^2
+\frac{1}{2}(-\widetilde{g}^{01}\widetilde{g}^{CD}\partial_0\widetilde{g}_{1D}+
    \partial_A\widetilde{g}^{AC})\partial_C\widetilde{g}_{01}
\end{equation*}
\begin{equation}\label{y36}
-\frac{1}{4}\widetilde{g}^{CB}(\partial_1\widetilde{g}_{CB}-\partial_0\widetilde{g}_{CB})
\partial_1\widetilde{g}_{01}+\frac{1}{2}\widetilde{g}^{DB}(\partial_B\widetilde{g}_{01})\widetilde{\Gamma}^C_{CD}.
\end{equation}
We do not give the expression of $\widetilde{R}_{0A}$ here since it
is not used directly in the construction of the constraints
equations. \section{Computation of $\frac{\partial}{\partial
   y^0}(\widetilde{\Gamma}^0+
    \widetilde{\Gamma}^1)$ on $\mathcal{C}$}\label{A4}
 By definition one has:
\begin{eqnarray*}
  \frac{\partial}{\partial
   y^0}(\widetilde{\Gamma}^0+
    \widetilde{\Gamma}^1) &=& \frac{\partial}{\partial
   y^0}[\widetilde{g}^{\mu\nu}(\widetilde{\Gamma}^0_{\mu\nu}+
   \widetilde{\Gamma}^1_{\mu\nu})] \\
   &=& (\frac{\partial}{\partial
   y^0}\widetilde{g}^{\mu\nu})(\widetilde{\Gamma}^0_{\mu\nu}+
   \widetilde{\Gamma}^1_{\mu\nu})+\widetilde{g}^{\mu\nu}\frac{\partial}{\partial
   y^0}(\widetilde{\Gamma}^0_{\mu\nu}+
   \widetilde{\Gamma}^1_{\mu\nu}),
   \end{eqnarray*}according to the properties
 of the trace of the metric and its inverse on the cone (\ref{c4}), (\ref{y4}), one
 has on $\mathcal{C}$
 \begin{equation*}
    \frac{\partial}{\partial
   y^0}(\widetilde{\Gamma}^0+
    \widetilde{\Gamma}^1)=2\widetilde{g}^{01}\frac{\partial}{\partial y^0}(\widetilde{\Gamma}^0_{01}+
   \widetilde{\Gamma}^1_{01})+\widetilde{g}^{11}\frac{\partial}{\partial y^0}
   (\widetilde{\Gamma}^0_{11}+
   \widetilde{\Gamma}^1_{11})+
 \end{equation*}
\begin{equation*}
 \widetilde{g}^{AB}\frac{\partial}{\partial y^0}
   (\widetilde{\Gamma}^0_{AB}+
   \widetilde{\Gamma}^1_{AB})+ (\frac{\partial \widetilde{g}^{\mu\nu}}{\partial
   y^0})(\widetilde{\Gamma}^0_{\mu\nu}+
   \widetilde{\Gamma}^1_{\mu\nu}).
\end{equation*}Furthermore,
\begin{equation*}
    \frac{\partial}{\partial y^0}(\widetilde{\Gamma}^0_{01}+
   \widetilde{\Gamma}^1_{01}) =
\end{equation*}
\begin{eqnarray*}
  && \frac{1}{2} \frac{\partial}{\partial y^0}
   [\widetilde{g}^{0\mu}(\partial_0 \widetilde{g}_{1\mu}+\partial_1\widetilde{g}_{0\mu}
   -\partial_\mu \widetilde{g}_{01})+
   \widetilde{g}^{1\mu}(\partial_0 \widetilde{g}_{1\mu}+\partial_1\widetilde{g}_{0\mu}
   -\partial_\mu \widetilde{g}_{01})]\\
   &=&\frac{1}{2}(\frac{\partial (\widetilde{g}^{0\mu}+\widetilde{g}^{1\mu})}{\partial
   y^0})(\partial_0 \widetilde{g}_{1\mu}+\partial_1\widetilde{g}_{0\mu}
   -\partial_\mu \widetilde{g}_{01})+\\
   &&\frac{1}{2}(\widetilde{g}^{0\mu}+\widetilde{g}^{1\mu})
   (\partial^2_{00} \widetilde{g}_{1\mu}+\partial^2_{01}\widetilde{g}_{0\mu}
   -\partial^2_{0\mu} \widetilde{g}_{01}),
\end{eqnarray*}using the expressions of the trace of the metric and its inverse
(\ref{c4}), (\ref{y4}), the relations (\ref{m31}) of appendix
\ref{A1}, this expression simplifies:
\begin{equation}
  \frac{\partial}{\partial y^0}(\widetilde{\Gamma}^0_{01}+
   \widetilde{\Gamma}^1_{01})= -\frac{1}{2}(\widetilde{g}^{01})^2
   \partial_0\widetilde{g}_{01}\partial_1\widetilde{g}_{01}+
   \frac{1}{2}\widetilde{g}^{01}\partial^2_{01}\widetilde{g}_{00}.
\end{equation}The term
 $ \frac{\partial}{\partial y^0}(\widetilde{\Gamma}^0_{11}+
   \widetilde{\Gamma}^1_{11})$ in turn reads:
   \begin{eqnarray*}
  \frac{\partial}{\partial y^0}(\widetilde{\Gamma}^0_{11}+
   \widetilde{\Gamma}^1_{11}) &=& \frac{1}{2}(\frac{\partial (\widetilde{g}^{0\mu}+\widetilde{g}^{1\mu})}{\partial
   y^0})(2\partial_1 \widetilde{g}_{1\mu}
   -\partial_\mu \widetilde{g}_{11})+\\
   &&\frac{1}{2}(\widetilde{g}^{0\mu}+\widetilde{g}^{1\mu})
   (2\partial^2_{01} \widetilde{g}_{1\mu}
   -\partial^2_{0\mu} \widetilde{g}_{11}).
\end{eqnarray*}After simplifications thanks to the same arguments as above,
 this expression results in:
 \begin{equation}
\frac{\partial}{\partial y^0}(\widetilde{\Gamma}^0_{11}+
   \widetilde{\Gamma}^1_{11})=-\frac{1}{2}(\widetilde{g}^{01})^2
   \partial_0\widetilde{g}_{01}(2\partial_1\widetilde{g}_{01}-\partial_0\widetilde{g}_{11})+
   \frac{1}{2}\widetilde{g}^{01}
   (2\partial^2_{01}\widetilde{g}_{01}-\partial^2_{00}\widetilde{g}_{11}).
 \end{equation}
Concerning the terms $ \frac{\partial}{\partial
y^0}(\widetilde{\Gamma}^0_{AB}+
   \widetilde{\Gamma}^1_{AB})$, computations and various
   simplifications result in:
   \begin{equation}
\frac{\partial}{\partial y^0}(\widetilde{\Gamma}^0_{AB}+
   \widetilde{\Gamma}^1_{AB})=\frac{1}{2}(\widetilde{g}^{01})^2\partial_0\widetilde{g}_{01}
   (\partial_0\widetilde{g}_{AB})+
   \frac{1}{2}\widetilde{g}^{10}(\partial^2_{0A}\widetilde{g}_{0B}+
   \partial^2_{0B}\widetilde{g}_{0A}-\partial^2_{00}\widetilde{g}_{AB}).
   \end{equation}Due to the fact that on the cone one has:
    $\partial_0\widetilde{g}_{0A}=0,\;
    \partial_0\widetilde{g}_{00}=\partial_0\widetilde{g}_{01}$, it
    follows also that
     $\partial^2_{0A}\widetilde{g}_{0B}=0=
   \partial^2_{0B}\widetilde{g}_{0A},\;\partial^2_{01}\widetilde{g}_{00}-
   \partial^2_{01}\widetilde{g}_{01}=0$, and therefore at this step one has:
\begin{equation}\label{x39}
   \frac{\partial}{\partial
   y^0}(\widetilde{\Gamma}^0+
    \widetilde{\Gamma}^1)=\left\{
       \begin{array}{ll}
         (\widetilde{g}^{01})^2\partial^2_{00}\widetilde{g}_{11}
-\frac{1}{2}\widetilde{g}^{01}\widetilde{g}^{AB}(\partial^2_{00}\widetilde{g}_{AB})
    -\frac{1}{2}(\widetilde{g}^{01})^3\partial_0\widetilde{g}_{01}
   \partial_0\widetilde{g}_{11}+& \hbox{} \\
   \frac{1}{2}
   (\widetilde{g}^{01})^2\widetilde{g}^{AB}
   \partial_0\widetilde{g}_{01}\partial_0\widetilde{g}_{AB}
  +(\frac{\partial \widetilde{g}^{\mu\nu}}{\partial
   y^0})(\widetilde{\Gamma}^0_{\mu\nu}+
   \widetilde{\Gamma}^1_{\mu\nu}). & \hbox{}
       \end{array}
     \right.
\end{equation}Now, we are interested of the term
 $X_1$:
\begin{equation*}
    X_1=(\frac{\partial \widetilde{g}^{\mu\nu}}{\partial
   y^0})(\widetilde{\Gamma}^0_{\mu\nu}+
   \widetilde{\Gamma}^1_{\mu\nu}).
\end{equation*}  Although this term does not contain second order outgoing derivatives,
it is important to highlight in it the presence of the term
$\partial_0 \widetilde{g}_{01}$. In virtue of the expressions of the
Christoffel symbols of the metric on $\mathcal{C}$ of appendix
\ref{A2}, the following computations hold on $\mathcal{C}$:
\begin{equation*}
\widetilde{\Gamma}^0_{00}+\widetilde{\Gamma}^1_{00}=
\frac{1}{2}\widetilde{g}^{10}\partial_0\widetilde{g}_{00},\;
\widetilde{\Gamma}^0_{01}+\widetilde{\Gamma}^1_{01}=
\frac{1}{2}\widetilde{g}^{10}\partial_1\widetilde{g}_{00},
\end{equation*}
\begin{equation*}
\widetilde{\Gamma}^0_{11}+\widetilde{\Gamma}^1_{11}=
\frac{1}{2}\widetilde{g}^{10}(2\partial_1\widetilde{g}_{01}-\partial_0\widetilde{g}_{11}),\;
\widetilde{\Gamma}^0_{0A}+\widetilde{\Gamma}^1_{0A}=
\frac{1}{2}\widetilde{g}^{10}\partial_A\widetilde{g}_{00},
\end{equation*}
\begin{equation*}
\widetilde{\Gamma}^0_{1A}+\widetilde{\Gamma}^1_{1A}=
\frac{1}{2}\widetilde{g}^{10}(\partial_A\widetilde{g}_{10}-\partial_0\widetilde{g}_{1A}
),\; \widetilde{\Gamma}^0_{AB}+\widetilde{\Gamma}^1_{AB}=
-\frac{1}{2}\widetilde{g}^{10}\partial_0\widetilde{g}_{AB}.
\end{equation*}Exploiting these expressions, the properties of the
trace of the metric and its inverse (\ref{c4}), (\ref{y4}), some
relations of (\ref{m31}) (appendix \ref{A1}), $X_1$ expresses as:
\begin{equation*}
    X_1=\left\{
       \begin{array}{ll}
      - \frac{1}{2}(\widetilde{g}^{01})^3 \partial_0\widetilde{g}_{11}
      \partial_0\widetilde{g}_{00}
+\frac{1}{2}(\widetilde{g}^{01})^3( \partial_0\widetilde{g}_{01}-
\partial_0\widetilde{g}_{11})
(2\partial_1\widetilde{g}_{01}-
\partial_0 \widetilde{g}_{11})& \hbox{} \\
    -( \widetilde{g}^{01})^3
(\partial_0
\widetilde{g}_{10}-\partial_0\widetilde{g}_{11})\partial_1
\widetilde{g}_{00}-(\widetilde{g}^{01})^2\widetilde{g}^{AC}\partial_0\widetilde{g}_{1C}
\partial_A\widetilde{g}_{00}+& \hbox{} \\
(\widetilde{g}^{01})^2\widetilde{g}^{AC}\partial_0\widetilde{g}_{1C}
(\partial_A\widetilde{g}_{01}-\partial_0\widetilde{g}_{1A})+
\frac{1}{2}\widetilde{g}^{01}\widetilde{g}^{AC}\widetilde{g}^{BD}
\partial_0\widetilde{g}_{AB}
\partial_0\widetilde{g}_{CD}, & \hbox{}
       \end{array}
     \right.
\end{equation*}and finally simplifies to:
 \begin{equation}\label{x38}
    X_1=\left\{
       \begin{array}{ll}
      - (\widetilde{g}^{01})^3 \partial_0\widetilde{g}_{01}\partial_0\widetilde{g}_{11}
+\frac{1}{2}(\widetilde{g}^{01})^3
(\partial_0\widetilde{g}_{11})^2& \hbox{} \\
  -(\widetilde{g}^{01})^2\widetilde{g}^{AC}\partial_0\widetilde{g}_{1C}
\partial_0\widetilde{g}_{1A}+
\frac{1}{2}\widetilde{g}^{01}\widetilde{g}^{AC}\widetilde{g}^{BD}
\partial_0\widetilde{g}_{AB}
\partial_0\widetilde{g}_{CD}.
 & \hbox{}
       \end{array}
     \right.
\end{equation}Combining the expressions (\ref{x39}), (\ref{x38}),
one ends up by:
    \begin{equation*}
         \frac{\partial}{\partial
   y^0}(\widetilde{\Gamma}^0+
    \widetilde{\Gamma}^1)=\frac{1}{2}(\widetilde{g}^{01})^2\partial^2_{00}\widetilde{g}_{11}
-\frac{1}{2}\widetilde{g}^{01}\widetilde{g}^{AB}(\partial^2_{00}\widetilde{g}_{AB})
    -\frac{3}{2}(\widetilde{g}^{01})^3\partial_0\widetilde{g}_{01}
   \partial_0\widetilde{g}_{11}+
    \end{equation*}
    \begin{equation*}
        \frac{1}{2}(\widetilde{g}^{01})^3
(\partial_0\widetilde{g}_{11})^2+
\frac{1}{2}(\widetilde{g}^{01})^2\partial_0\widetilde{g}_{01}\widetilde{g}^{AB}\partial_0
\widetilde{g}_{AB}-
    \end{equation*}
\begin{equation}\label{y35}
   (\widetilde{g}^{01})^2\widetilde{g}^{AC}\partial_0\widetilde{g}_{1C}
\partial_0\widetilde{g}_{1A}+
\frac{1}{2}\widetilde{g}^{01}\widetilde{g}^{AC}\widetilde{g}^{BD}
\partial_0\widetilde{g}_{AB}
\partial_0\widetilde{g}_{CD}.
\end{equation}
\end{appendices}{}

\textbf{Patenou Jean Baptiste}\\
 Department of Mathematics and Computer Science,\\
Faculty of Science, University of Dschang, Cameroon, P. O. Box. 67,
Dschang. E-mail: jeanbaptiste.patenou@univ-dschang.org,
jpatenou@yahoo.fr
\end{document}